\documentclass[12pt,a4paper]{amsart}
\usepackage[letterpaper,margin=1in]{geometry}

\usepackage{graphicx}
\usepackage{algorithm}
\usepackage{algorithmic}

\usepackage{hyperref}
\usepackage{amsmath}
\usepackage{amssymb}
\usepackage{mathtools}
\usepackage{amstext}
\usepackage{verbatim}
\usepackage{graphicx}
\usepackage{subcaption}
\usepackage{algorithm}
\usepackage{algorithmic}
\usepackage[colorinlistoftodos]{todonotes}
\usepackage{xpatch}

\numberwithin{equation}{section}
\newtheorem{theorem}{Theorem}[section]
\newtheorem{corollary}[theorem]{Corollary}
\newtheorem{lemma}[theorem]{Lemma}

\newtheorem{remark}[theorem]{Remark}

\newcommand{\rset}{\mathbb{R}}
\newcommand{\cset}{\mathbb{C}}

\newcommand{\e}{\mathtt{e}}
\newcommand{\tz}{\widetilde{z}}
\DeclareMathOperator*{\argmin}{arg\,min}

\def\R{\mathbb{R}}

\def\VIZ#1{(\ref{#1})}

\def\ALPHAEX{\gamma} 
\def\rmD{{\mathrm{d}}}
\def\BARS{\mathbf{S}} %

\def\IU{\mathrm{i}}
\def\BFO{\boldsymbol{\omega}}

\def\PERIOD{\,.}

\def\RE{\mathrm{Re}}

\def\GF{\beta}
\def\GFO{\beta^*}
\def\rLK{\frac{1}{KL}}

\def\rJ{\frac{1}{J}}
\def\UBB{{\bar{\bar{u}}}}
\def\CTRL{\chi}

\begin{document}

\title[Accuracy of deep residual neural networks]{Smaller generalization error derived for a deep residual neural network compared 
to shallow networks
}


\author[Kammonen]{Aku Kammonen} 
\address{KTH Royal Institute of Technology, Stockholm, Sweden}

\author[Kiessling]{Jonas Kiessling}
\address{RWTH Aachen University, Aachen, Germany and KTH Royal Institute of Technology, Stockholm, Sweden}

\author[Plech\'a\v{c}] {Petr Plech\'a\v{c}}
\address{University of Delaware, Newark, USA \\}

\author[Sandberg]{        Mattias Sandberg}
\address{KTH Royal Institute of Technology, Stockholm, Sweden}

\author[Szepessy]{        Anders Szepessy}
\address{KTH Royal Institute of Technology, Stockholm, Sweden}

\author[Tempone]{        Ra\'ul Tempone}
\address{RWTH Aachen University, Aachen, Germany and KAUST, Saudi Arabia}





\maketitle

\begin{abstract}
{Estimates of the generalization error are proved for a residual neural network with $L$ random Fourier features layers
 $\bar z_{\ell+1}=\bar z_\ell + \RE\sum_{k=1}^K\bar b_{\ell k}e^{\IU\omega_{\ell k}\bar z_\ell}+
\RE\sum_{k=1}^K\bar c_{\ell k}e^{\IU\omega'_{\ell k}\cdot x}$. An optimal distribution for the frequencies $(\omega_{\ell k},\omega'_{\ell k})$ 
of the random Fourier features $e^{\IU\omega_{\ell k}\bar z_\ell}$ and $e^{\IU\omega'_{\ell k}\cdot x}$ is derived. This derivation is based on the corresponding generalization error for the approximation of the  function values $f(x)$. The generalization error turns out to be smaller than the estimate 
${\|\hat f\|^2_{L^1(\rset^d)}}/{(KL)}$ of the generalization error for random Fourier features with one hidden layer and the same total number of nodes $KL$, in the case the $L^\infty$-norm of $f$ is much less than the $L^1$-norm of its Fourier transform $\hat f$. 
This understanding of an optimal distribution for random features is used to construct a new training method for a deep residual network. Promising performance of the proposed new algorithm is demonstrated in computational experiments.}
\end{abstract}



\section{Introduction}

\subsection{Residual neural network formulation}
We study approximation properties of deep residual neural networks applied
to  supervised learning with a given data set 
$\{(x_n,y_n)\in \rset^d\times \rset \ | \ n=1,\ldots,N\}$, where the $x_n$ are independent samples from an unknown probability distribution.
We consider noisy data $y_n=f(x_n)+\epsilon_n$, for  functions $f:\rset^d\to \rset$, with their Fourier transform bounded in $L^1(\rset^d)$. 
The noise $\{\epsilon_n\}_{n=1}^N$ are independent and identically distributed (iid) samples, also independent of the data, with finite second moment $\epsilon^2$. 
The objective is to use the data  to determine a neural network
approximation that approximates the function $f$ as accurately as possible.
We focus on a particular case of a {\it residual} neural network with the network activation function $s(\omega,x)=e^{\IU \omega\cdot x}$.
We denote $\omega\cdot x=\sum_{i=1}^d \omega^i x^i$ the standard
Euclidean product on $\R^d$.
In effect we construct a multi-layer (deep) {\it random Fourier feature} approximation.

To motivate the use of the random Fourier features in the presented analysis 
we first introduce a neural
network with a single hidden layer with $K'$ nodes 
and the activation function $s(\omega,x)=e^{\IU \omega\cdot x}$.
The related standard neural network approximation problem, cf. \cite{understand},
requires to find
the amplitudes $\bar c_k\in \cset$ and frequencies $\omega'_k\in\rset^d$, $k=1,\dots,K'$,
which minimize  the risk functional, i.e., solving the problem
\[
\min_{\substack{(\omega_k',\bar c_k)\in\rset^{d}\times \cset\\ k=1,\ldots,K'}}
N^{-1}\sum_{n=1}^N|y_n-\RE \sum_{k=1}^{{K'}} {\bar c_{k}e^{\IU\omega'_{k}\cdot x_n}} |^2
\]
which for $N\to\infty$ becomes the generalization error
\begin{equation}
\min_{\substack{(\omega_k',\bar c_k)\in\rset^{d}\times \cset\\ k=1,\ldots,K'}}
\mathbb E_{xy}[|y-\RE \sum_{k=1}^{{K'}} {\bar c_{k}e^{\IU\omega'_{k}\cdot x}} |^2] \PERIOD
\label{min_one_layer}
\end{equation}
We denote by $\mathbb E_{xy}$ the expected value with respect to the $\{(x_n,y_n)\, |\, n=1,2,3,\ldots\}$ data distribution and by
$\mathbb E_\omega$ the expected value with respect to the frequency distribution $\{\omega_k,\omega_k')\,|\, k=1,2,\ldots\}$.

Here, we relax problem \VIZ{min_one_layer} by studying a more tractable problem, namely we assume instead that the frequencies 
$\{\omega_{k}'\in\rset^d\ |\ k=1,\ldots, K'\}$ 
are 
iid random variables with a certain distribution
$\bar p':\rset^d\to [0,\infty)$. Then
the generalization error for a random Fourier feature neural network with one hidden layer 
and $K'$ nodes has the bound
\begin{equation}
   \mathbb E_\omega\big[\min_{\bar c\in \cset^{K'}}
\mathbb E_{xy}[|f(x)-\RE \sum_{k=1}^{{K'}} {\bar c_{k}e^{\IU\omega'_{k}\cdot x}} |^2] \big]
\le \frac{1}{K'} \int_{\rset^d}\frac{|\hat f(\omega)|^2}{\bar p'(\omega)} \rmD \omega
\,,
\label{one_layer}
\end{equation}
where we introduce the Fourier transform 
\begin{equation}
\hat f(\omega):=\frac{1}{(2\pi)^{d}}\int_{\rset^d}f(x)e^{-\IU \omega\cdot x} \rmD x
\label{FTdef}
\end{equation}
and its  inverse representation 
\[
f(x)=\int_{\rset^d}\hat f(\omega)e^{\IU\omega\cdot x} \rmD x\,.
\]
The form \VIZ{FTdef} of the Fourier transform  is convenient for presenting various formulas without $2\pi$ factors.
The well known proof of \VIZ{one_layer}, cf. \cite{barron} and \cite{jones}, 
is based on Monte Carlo quadrature of this inverse Fourier representation,  see Lemma~\ref{MC}.

 The estimate \VIZ{one_layer} indicates the possibility to improve the bound by approximating the minimizing density  $\bar p'=|\hat f|/\|\hat f\|_{L^1(\rset^d)}$ (cf. Lemma~\ref{p_opt}), for instance by the adaptive Metropolis method in \cite{ARFM}.
Since a minimum is less than or equal to its corresponding mean, \VIZ{one_layer} implies
\begin{equation}
\begin{split}
\min_{\substack{(\bar c_k,\omega_k')\in\cset\times\rset^{d}\\ k=1,\ldots, K'}}
\mathbb E_{xy}[|f(x)-\RE \sum_{k=1}^{{K'}} {\bar c_{k}e^{\IU\omega'_{k}\cdot x}} |^2] 
&\le
\min_{\bar p'}\mathbb E_\omega\big[\min_{\bar c\in \cset^{K'}}
\mathbb E_{xy}[|f(x)-\RE \sum_{k=1}^{{K'}} {\bar c_{k}e^{\IU\omega'_{k}\cdot x}} |^2] \big]\\
&\le \frac{ \|\hat f\|_{L^1(\rset^d)}^2}{K'}\,,
\end{split}
\label{min_mean}
\end{equation}
which provides an error estimate also for convergent gradient or stochastic gradient approximations of the optimization problem on the left hand side.


The aim of this work is to generalize random Fourier features to include several hidden layers
and study the corresponding generalization error.
In particular we prove in Theorem~\ref{thm_main} that the generalization error 
for a deep residual neural network
can be smaller than the  bound  \VIZ{min_mean} 
for neural networks with one hidden layer and the same number of nodes.
This approximation result also provides an optimal parameter distribution,
which we use to formulate a generalization 
of the adaptive Metropolis method \cite{ARFM} to deep random features networks.

\smallskip

We study deep residual neural networks,  with $L$ layers, $K$ nodes on the initial layer and $2K$ nodes on the subsequent residual layers.
We define the single layer neural network function at the layer $\ell=0$
\begin{equation}\label{Gfundef}
\GF(x):= \RE \sum_{k=1}^{{K}} {\bar c_{0 k}e^{\IU\omega'_{0 k}\cdot x}}\,,
\end{equation}
and represent residuals at subsequent layers by
\begin{equation*} 
\begin{split}
 \bar z_{\ell+1}(x) & =\bar z_\ell(x) + \RE\sum_{k=1}^K\bar b_{\ell k}e^{\IU\omega_{\ell k}\bar z_\ell(x)}+
\RE\sum_{k=1}^K\bar c_{\ell k}e^{\IU\omega'_{\ell k}\cdot x}\,,\;\;\;
\ell=1,\ldots,L-1\, \\
\bar z_1 & = 0 \,.
\end{split}
\end{equation*}
For simplicity we denote this as $L$ layers, even though one would typically denote this a $2L$ layer network. 
The output of the network becomes
$$
\bar z_L(x) + \GF(x) \,.
$$
Note that there is no connection between layer $\ell=0$ and layers $\ell>0$, that is $\bar z_{\ell}(x)$ does not have an explicit dependence on $\GF(x)$. However, the parameters for $\bar z_{\ell}(x)$ depend on $\GF(x)$, as we will see below.

The residual neural network is thus defined by the parameters $\theta_{\ell k}=(\bar b_{\ell k},\omega_{\ell k})$, $\theta'_{\ell k}=(\bar c_{\ell k},\omega'_{\ell k})$, $\ell=0,\ldots,L-1$, $k=1,\ldots,K$ that minimize the generalization error
\begin{equation}\label{min_gerror0}
\begin{split}
&\min_{\substack{\theta_{0k}'\\ k=1,\ldots,K}} \mathbb E_{xy}
\Big[|y-\RE \sum_{k=1}^{{K}} {\bar c_{0k}e^{\IU\omega'_{0 k}\cdot x}} |^2\Big]\,,\;\;\mbox{and}\\
&\min_{\substack{
\theta_{\ell k}, \theta'_{\ell k}\\ 
\ell=1,\ldots, L-1 \\ 
k=1,\ldots,{{K}}
}} \mathbb E_{xy}\Big[|{\bar z}_L-\big(y-\GF(x)\big)|^2 + \bar\delta L \sum_{\ell=1}^{L-1}|\bar z_{\ell+1}-\bar z_\ell|^2\Big]\,,\\
&\mbox{subject, for $\ell=1,\ldots,L-1$, to } \\
&{\bar z}_{\ell +1} = {\bar z}_\ell + \RE \sum_{k=1}^{{K}} \bar b_{\ell k}e^{\IU \,  \omega_{\ell k} {\bar z}_\ell} 
+\RE \sum_{k=1}^{{K}} \bar c_{\ell k}e^{\IU\omega'_{\ell k}\cdot x}\,,
\\ 
&{\bar z}_1 = 0\,.
\end{split}
\end{equation}

Rather than solving the non-convex minimization problem \VIZ{min_gerror0}
we follow the strategy motivated by the single layer case explained above.
We use random Fourier features, i.e., choose random frequencies
$\omega_{\ell k}$, $\omega'_{\ell k}$ and solve first the
convex minimization
problem for the parameters $\bar c_{0 k }$
\begin{equation}\label{min_ls}
\min_{\substack{\bar c_{0k}\\ k=1,\ldots,K}}
\mathbb E_{xy}[|y-\RE \sum_{k=1}^{{K}} {\bar c_{0k}e^{\IU\omega'_{0 k}\cdot x}} |^2]\,,
\end{equation}
with the solution $\bar c^*_{0k}$, $k=1,\ldots,K$ of
\VIZ{min_ls} defining
the single layer approximation
\begin{equation}\label{betastar}
\GFO(x)= \RE \sum_{k=1}^{{K}} {\bar c^*_{0 k}e^{\IU\omega'_{0 k}\cdot x}}\,,
\end{equation}
which in turn defines the optimization problem for the deep residual network, i.e., for 
$(\bar b_{\ell k}, \bar c_{\ell k})$, $\ell=1,\ldots,L-1$ 
\begin{equation}\label{min_barz}
\begin{split}
&\min_{\substack{
(\bar b_{\ell k}, \bar c_{\ell k})\in \cset\times\cset\\ \ell=1,\ldots, L-1 \\ k=1,\ldots,K
}}
\mathbb E_{xy}\big[|{\bar z}_L-\big(y-\GFO(x)\big)|^2 + \bar\delta L \sum_{\ell=1}^{L-1}|\bar z_{\ell+1}-\bar z_\ell|^2\big]\,,\\
&\mbox{subject, for $\ell=1,\ldots,L-1$, to } \\
&{\bar z}_{\ell +1} = {\bar z}_\ell + \RE \sum_{k=1}^{{K}} \bar b_{\ell k}e^{\IU \,  \omega_{\ell k} {\bar z}_\ell}
+\RE \sum_{k=1}^{{K}} \bar c_{\ell k}e^{\IU\omega'_{\ell k}\cdot x}\,,\\
&{\bar z}_1 = 0\,.
\end{split}
\end{equation}
The purpose of the initial optimization \VIZ{min_ls} is to obtain an optimization problem for $\bar z_\ell$ with smaller data $y-\GFO(x)$ instead of $y$, which turns out to be useful in our proof and experiments; to simplify, the reader may consider $\beta=\beta^*=0$, except in \eqref{beta_zero}.
The non negative number $\bar\delta$ is a  regularization parameter. 
We include the penalty term, $\bar\delta L \sum_{\ell=0}^{L-1}|\bar z_{\ell+1}-\bar z_\ell|^2$, 
since the penalty has a role in finding an optimal solution
as described in Section~\ref{sec_2}.
The random frequencies $\omega_{\ell k}$, $\omega'_{\ell k}$ are independent identically distributed random variables for $\ell=0,\ldots, L-1$ and $k=1,\ldots, K$ and are sampled from yet to be determined distributions with densities depending on the frequencies and layers. 

In the case that all $\bar b_{\ell k}=0$, the formulation \VIZ{min_barz} is equivalent to approximation with  a random Fourier feature network with one hidden layer and $KL$ nodes.
The purpose of the presented analysis is to derive  estimates of the generalization error for \VIZ{min_barz}
that improves the bound \VIZ{min_mean}, with $K'=KL$,
and determine  optimal choices of $L,K,\bar p,\bar p',\bar\delta$ that minimize the generalization error for \VIZ{min_barz}.

%
%

Our analysis is based on viewing the optimization problem \VIZ{min_barz} as a discrete version of a continuous
time optimal control problem. More precisely, in our analysis we use the solution to an optimal control problem 
related to \VIZ{min_barz} with infinite number of nodes  and layers. 
Details of the analysis and the complete proof are provided in Sections~\ref{sec_2} and \ref{sec_proof}. 
Here, we outline only a basic
connection between the optimization problem \VIZ{min_barz}, viewed as a discrete optimal control, and its 
continuum time counter part. We postpone the detailed discussion to Section~\ref{sec_2}.

\smallskip

For simplicity,  in this section we assume that the data is noiseless, i.e., $\epsilon = 0$,  and $y_n=f(x_n)$ for some (unknown) function $f$.
We introduce time-dependent controls $\hat b:[0,1]\times\rset\to\cset$ and
$\hat c:[0,1]\times\rset^d\to\rset$. We define the control
function $\alpha:[0,1]\times\rset\times\rset^d\to \rset$
\begin{equation}\label{def_alpha}
\alpha(t,z;x) := \RE \int_{\rset} \hat b(t,\omega)e^{\IU \omega z_t} \rmD\omega 
                  + \RE\int_{\rset^d} \hat c(t,\omega')e^{\IU \omega'\cdot x}\rmD\omega'\,,
\end{equation}
and formulate the optimal control problem with the time-dependent 
state function $z_t(x)$ with its dynamics defined by $\alpha(t,z;x)$.
Thus we solve the optimal control problem
\begin{equation}\label{min_z1}
\begin{split}
&\min_{\substack{
\hat b:[0,1]\times\rset\to\mathbb C\\
\hat c:[0,1]\times\rset^d\to \mathbb C}
}
\mathbb E_{x}\big[| z_1-\big(f(x)-\GF(x)\big)|^2 + \delta\int_0^1 |\alpha(t,z_t,x)|^2 \rmD t\big] \,,\\
&\mbox{subject to }\;\; \frac{{\rmD} z_t}{{\rmD}t} = \alpha(t,z_t;x) 
\\
 &z_0 = 0\,,
\end{split}
\end{equation}
where $\GF(x)$ is defined by \VIZ{Gfundef}.
The optimal control problem \VIZ{min_z1} can be solved explicitly as described in Lemma~\ref{lemma_z}. The explicit solution is then used to derive the bounds \VIZ{C1} and \VIZ{pp'} in Theorem~\ref{thm_main}, 
see Section~\ref{sec_2}.

The relation to the optimization problem \VIZ{min_barz} is obtained
by using Monte Carlo approximation of integrals over $\omega$, $\omega'$ and $t$,
thus introducing random
times and frequencies $(t_{\ell k},\omega_{\ell k})\sim p(t,\omega) \rmD t\,\rmD\omega$,
$(t'_{\ell k},\omega'_{\ell k})\sim p'(t',\omega') \rmD t'\,\rmD\omega'$
in 
\[
\alpha(t,z_t;x) =  \RE \int_{\rset} \hat b(t,\omega)e^{\IU \omega z_t} \rmD\omega 
                  + \RE\int_{\rset^d} \hat c(t,\omega')e^{\IU \omega'\cdot x}\rmD\omega'\,,
\]
which gives the Monte Carlo approximations
\begin{eqnarray}\label{omega_t_MC}
    \sum_{\ell=0}^{L-1}\int_{t_{\ell}}^{t_{\ell+1}}\int_{\rset} \hat b(t,\omega)e^{\IU \omega z_t}\rmD t \rmD\omega &\approx& \rLK \sum_{\ell=0}^{L-1}\sum_{k=1}^K\frac{\hat b(t_{\ell k}, \omega_{\ell k}) e^{\IU\omega_{\ell k}\tz_\ell}}{p(t_{\ell k},\omega_{\ell k})}\,, \\
    \sum_{\ell=0}^{L-1}\int_{t_{\ell}}^{t_{\ell+1}}\int_{\rset^d} \hat c(t',\omega')e^{\IU \omega'\cdot x}\rmD t'\rmD\omega' &\approx & \rLK\sum_{\ell=0}^{L-1}
    \sum_{k=1}^K \frac{ \hat c(t'_{\ell k},\omega'_{\ell k}) e^{\IU\omega'_{\ell k}\cdot x}}{p'(t'_{\ell k},\omega'_{\ell k})}\,,
\end{eqnarray}
with $\tz_\ell:= \tz(t_\ell)$ denoting the state dynamics $z_t$ evaluated at times 
$t_\ell = Q^{-1}(\tfrac{\ell}{L})$,
where $Q:[0,1]\to[0,1]$ is the cumulative distribution function for a given density $q:[0,1]\to [0,\infty)$.
The random times and random frequencies are independent, therefore the distributions take the form
$p(t,\omega) = \bar p(\omega) q(t)$ and $p'(t',\omega') = \bar p'(\omega') q'(t')$,
$t_\ell = Q^{-1}(\tfrac{\ell}{L}))$.
Thus, heuristically, in the limit $L\to\infty$ and $K\to\infty$ 
\[
\bar b_{\ell k}\simeq \rLK \frac{\hat b(t_{\ell k}, \omega_{\ell k})}{p(t_{\ell k},\omega_{\ell k})}\;\;
\mbox{ and }\;\; \bar c_{\ell k}\simeq \rLK \frac{ \hat c(t'_{\ell k},\omega'_{\ell k})}{p'(t'_{\ell k},\omega'_{\ell k})}\,,
\]
and the optimal solution of the  problem \VIZ{min_barz} can be linked to the approximation of
the optimal control solution of \VIZ{min_z1}. The analysis presented in Section \ref{sec_2} clarifies this connection rigorously and
a precise formulation of $t_{\ell k},t'_{\ell k},$ $\omega_{\ell k},\omega'_{\ell k}$, including stratified sampling for $t_{\ell k},t'_{\ell k}$ related to the layers, is presented in Section \ref{MainRes}.

\begin{remark}
{\rm
The optimization problem  \VIZ{min_barz} can be generalized to use $d$-dimensional control functions 
${\bar z}_\ell\in\rset^d$ in which case we have  $(\bar b_{\ell k},\bar c_{\ell k})\in \cset^d\times \cset^d$.
For a given fixed vector $\e\in\R^d$ we then have an optimization problem analogous to \VIZ{min_barz}
\begin{equation}\label{min_barze}
\begin{split}
&\min_{{(\bar b_{\ell k}, \bar c_{\ell k})\in \cset^d\times\cset^d}}
\mathbb E_{xy}\big[|\e\cdot{\bar z}_L-\big(y-\GF(x)\big)|^2
+\bar\delta L\sum_{\ell=0}^{L-1}|\bar z_{\ell+1}-\bar z_\ell|^2\big]\,,\\
&\mbox{subject, for $\ell = 1,\ldots, L-1$, to }\\ &{\bar z}_{\ell +1} = {\bar z}_\ell + \RE \sum_{k=1}^{{K}} 
      \bar b_{\ell k}e^{\IU\omega_{\ell k}\cdot {\bar z}_\ell}
+\RE \sum_{k=1}^{{K}} \bar c_{\ell k}e^{\IU\omega'_{\ell k}\cdot x}\,,\\
&{\bar z}_1 = 0\, ,
\end{split}
\end{equation}
with a similar error estimate as for \VIZ{min_barz}. 
A common formulation is to let all $\bar c_{\ell k}=0$, see \cite{Weinan2018AMO}, where also a different regularization term is used.  Then the initial data is typically ${\bar z}_0=x$ instead of ${\bar z}_0=0$.  The vector $\e$ can also be replaced by another vector and the dimension of the vectors could be different from $d$.
}
\end{remark}

\begin{remark}
{\rm 
The derivation of the error estimates here shows an advantage of including non zero coefficients $\bar c_{\ell k}$, which also can be motivated from the perspective of using a Markov control allowed to depend on both the state $\bar z_\ell$ and the  data $x$. The inclusion of such previous layers in the dynamics of the state in deep neural networks is studied in so-called deep dense neural networks, see \cite{dense}. Note that the control variables $\bar b_{\ell k}$ and $\bar c_{\ell k}$ only depend on the distribution of the data $\{(x_n,y_n)\, | \, n=\mathbb Z_+\}$ and not on the individual outcomes.
}
\end{remark}
%
%
\subsection{Relation to previous work}
The main inspiration of our work is the construction and optimal control analysis of deep residual neural networks in \cite{weinan_apriori}, which proves an error estimate of the generalization error including also the more demanding case with finite sets of data $(x,y)$.  
That error estimate does not improve the error estimate compared to approximation with a single hidden layer. Other important  previous results include the formulation of random Fourier features in \cite{rahimi_recht} and the optimal control perspective in  \cite{Weinan2018AMO}. 
The main new mathematical idea in our work is to identify a simple structure in a related infinite dimensional optimal control problem \VIZ{min_z1} and use it to identify the partition of the control into two parts:
one part depending on the data $x$ and the other part depending on the state $\bar z_\ell$; and then use the two parts to minimize the variance of the Monte Carlo quadrature present in deep residual neural networks.
The mathematical technique used in our work 
is a combination of standard Monte Carlo approximations adapted to the error analysis of approximations of differential equations and optimal control problems, 
where the approximation error is represented as an integral of the difference between the exact and approximate value functions along the exact solution path, as e.g., in  \cite{SDE_adapt}.

There are several results on improved approximation for deep neural networks as compared to shallow networks for certain functions, e.g., in \cite{telgarsky}, \cite{tegmark}, \cite{Schwab}.  
The work  \cite{allenzhu2019resnet} proves that also the generalization error obtained from the stochastic gradient method 
can be smaller for three layer neural networks, based on a residual neural network of the form \VIZ{min_barz}, compared to kernel methods using the same number of stochastic gradient descent steps. The result uses the rectifier linear unit activation function, $v\mapsto \max(0,v)$, instead of the Fourier activation function $v\mapsto e^{\IU v}$ used in this study. 
The example providing better approximation in neural networks with three layers is
based on approximated functions that are composition functions related to the neural network construction. 

Our Theorem~\ref{thm_main} also estimates the generalization error instead of the minimal error in a certain norm, as in \cite{telgarsky,tegmark,Schwab}.  The error estimate in Theorem~\ref{thm_main} is not based on functions related to  the compositions given by the neural network, as in \cite{allenzhu2019resnet},   instead on functions with $L^1$-bounded Fourier transform.
This theorem establishes insight on the optimal distribution of all parameters 
for the deep residual network \VIZ{min_barz}. The theorem and Remark \ref{rem_C} also shows that for functions $f$, with $\|f\|_{L^\infty(\rset^d)}\ll \|\hat f\|_{L^1(\rset^d)}$, 
supervised learning with deep networks could have less generalization error compared to networks with one hidden layer. 
Note that we study upper bounds here, so we cannot conclude that deep residual networks have smaller generalization errors than shallow: although the Monte Carlo approximation error we use is sharp, 
the minimization error can be smaller depending on the regularity of the $x$ sampling density, see \cite{bach}. Clearly other deep neural networks could have better approximation by other reasons than those studied here. The purpose of studying the special network \VIZ{min_barz} is that for this particular setting we can provide theoretical motivation on the approximation and its optimal parameter distribution.

\smallskip

We present the result in Section~\ref{MainRes} and split the proof together with the computational
demonstration in the subsequent five sections. Section~\ref{sec_1-2} provides 
background on Monte Carlo quadrature and Section~\ref{sec_2} presents the optimal control solution 
of the related problem to \VIZ{min_barz} with infinite number of layers $L$ and nodes per layer $K$. 
The lemmas in these two sections are then used in Section~\ref{sec_proof} which proves the main result
in Theorem~\ref{thm_main}.   Sections~\ref{sec_alg} and \ref{sec:numerics} present new numerical algorithms and experiments, using the obtained optimal random feature distribution, to approximate the deep residual network problem \VIZ{min_barz} based on finite amount of data and adaptive Metropolis sampling of the frequencies. In particular numerical tests confirm that deep residual networks can have smaller generalization error compared to networks with one hidden layer and the same total number of nodes. The final Section~\ref{sec_conclusion} includes a short summary of the work.

\section{Statement and discussion of the main result}\label{MainRes} 
We begin by stating the main result and present the proof and supporting lemmas in subsequent sections.
The result provides the optimal densities $p(t,\omega)\rmD t\rmD\omega$ and $p'(t',\omega')\rmD t'\rmD\omega'$ and 
sharp upper bounds on the generalization error of the studied residual neural network.
The usage of the optimal densities in time-frequency domains for Monte Carlo approximation of integrals
is one of the principal features, thus we first explain the construction of the Monte Carlo 
approximation based on a stratified sampling approach.

To each random frequency $\omega_{\ell k}$ we associate a random time $t_{\ell k}\in [0,1]$ with the purpose to construct
 Monte Carlo approximations of the integrals over $\rset^d\times[0,1]$ in \VIZ{omega_t_MC}, such that each layer $\ell$
has $K$ independent times $t_{\ell k}, \ k=1,\ldots,K$, where $K$ is fixed and non random. 
While the sampling of random frequencies $\omega_{\ell k}$ and $\omega'_{\ell k}$ is straightforward, for the random times 
$t_{\ell k}$ and $t'_{\ell k}$ we employ stratified sampling as follows. 
For a given density $q:[0,1]\to [0,\infty)$ we denote its cumulative distribution function $Q(t) := \int_0^t q(s)\,{\rm d}s$. Then, 
for each layer $\ell\in\{0,\ldots,L-1\}$
let  $\tau_{\ell k}$, $k=1,\ldots, K$ be independent and uniformly distributed random variables on $[\frac{\ell}{L},\frac{\ell+1}{L})$,  independent also of all $\omega_{\ell k}$ and all $\omega'_{\ell k}$. 
We define the random times $t_{\ell k}:=Q^{-1}(\tau_{\ell k})$ and the non-random time levels $t_\ell:=Q^{-1}(\frac{\ell}{L})$, see Figure \ref{fig:tau-t}.
Assuming that $t\in [t_{\ell},t_{\ell+1})$ the density for the random variable $t_{\ell k}$ is given by 
\[
\begin{split}
\mathbb{P}\big(t_{\ell k}\in[t,t+{\rm d}t) \big)
&=\mathbb{P}\big(\tau_{\ell k}\in\big[Q(t),Q(t+{\rm d}t)\big)\big)
=\frac{q(t){\rm d}t}{L^{-1}}=Lq(t){\rm d}t\,.
\end{split}
\]
We also define the space time density $p(t,\omega):=\bar p(\omega)q(t)$.
Similarly, we associate 
$t'_{\ell k}$ with the density $q':[0,1]\to [0,\infty)$  to $\omega'_{\ell k}$ and define  $p'(t,\omega):=\bar p'(\omega)q'(t)$.

\begin{figure}[t!]
\centering
\includegraphics[width=0.6\textwidth]{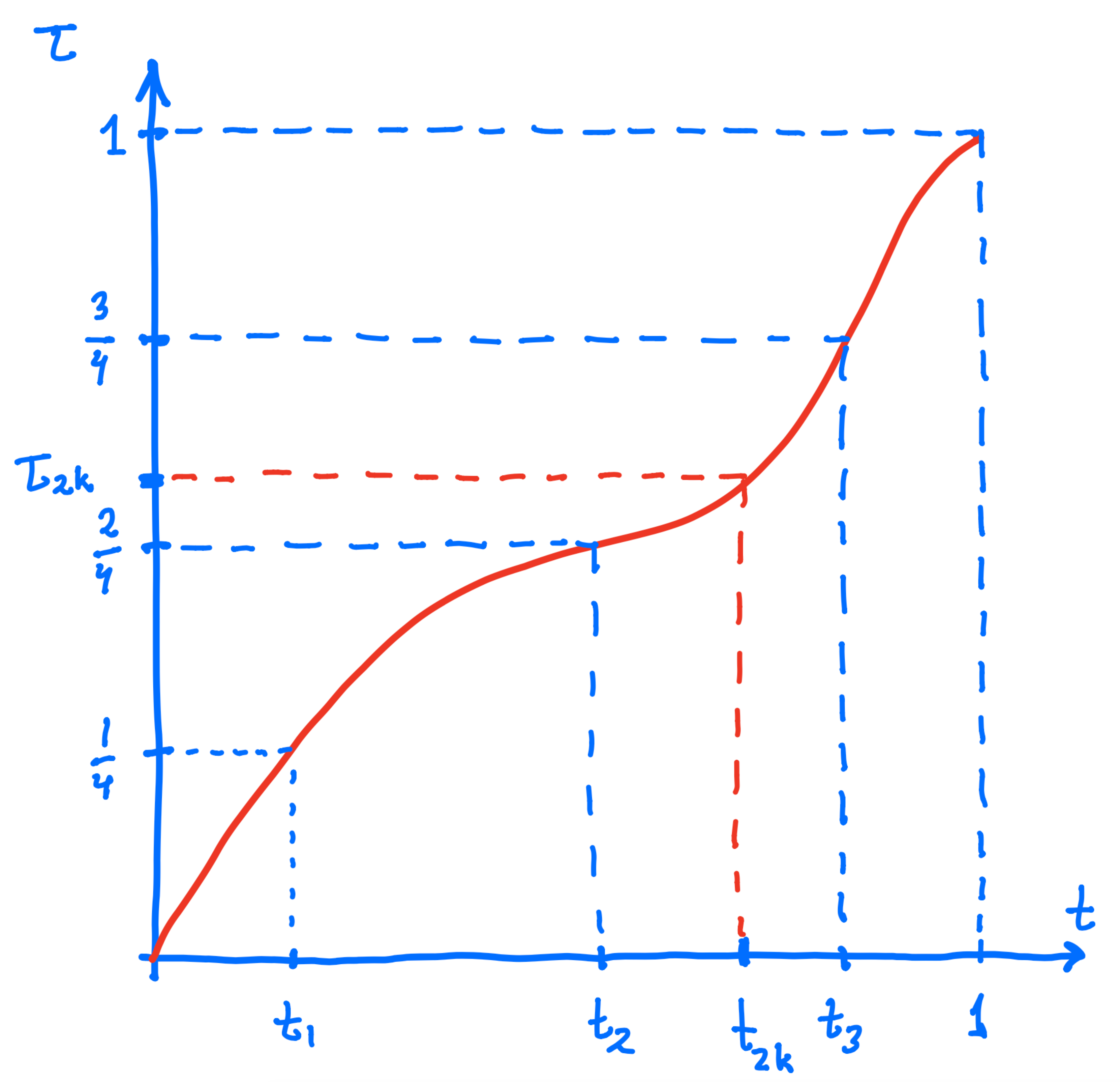}
\caption{Stratified sampling of $t_{\ell k}$ with four layers and the density proportional to $q$. The red curve is the graph  $Q:[0,1]\to[0,1]$.}\label{fig:tau-t}
\end{figure}

In order to state the main result with rigorous technical assumptions we select $h:\rset\to \rset$ to 
be a Schwartz function which is equal to one on $[-1,1]$, for example,
\[
h(z)=\left\{
\begin{array}{cl}
(1+e^{\frac{1}{1-z}})\, e^{-(1-z)^2} & \mbox{ for } z>1\,,\\
1 &\mbox{ for } |z|\le 1\,,\\
(1+e^{\frac{1}{1+z}})\, e^{-(1+z)^2} & \mbox{ for } z<-1\,.
\end{array}\right.
\]

\begin{theorem}\label{thm_main}
Assume that  the quantities 
\begin{equation*}
\begin{split}
&{F}:=\|f\|_{L^\infty(\rset^d)}\,,\\
&A  :=\int_{\rset^d}\frac{|\hat f(\omega)|^2}{\bar p'(\omega)} \rmD \omega\,,\\
&B' := {F}^{2}\int_{\rset}\frac{| \partial_\omega \hat h({F}\omega)|^2}{\bar p(\omega)} \rmD \omega\,,\\
&B := B' \exp\big(2\| \partial_z\big(zh(z/{F})\big)\|_{L^{\infty}(\rset)}\big)\, ,\\
&{F}\| \omega\partial_\omega \hat h({F}\omega)\|_{L^2(\rset)} +
\|\frac{| {F}\partial_\omega \hat h({F}\omega)|^2}{ \bar p(\omega)}\|_{L^{\infty}(\rset)}
+\| \partial^2_{z}\big(zh(z/{F})\big)\|_{L^{\infty}(\rset)}
+\|\frac{\hat f}{\bar p'}\|_{L^\infty(\rset^d)}\,, 
\end{split}
\end{equation*}
are bounded, ${\bar z^*}$ is an optimal solution to \VIZ{min_barz}, and $\GFO(x)$ an optimal solution to \VIZ{min_ls}.
Then there are positive constants $c$ and $C$ such that 
the generalization error satisfies
\begin{equation}\label{error_estimate}
\mathbb E_{t\omega}
\big[
\mathbb E_{xy}[|\bar z_L^*+\GFO(x) -y|^2]\big]
\le \frac{C}{KL} +\mathcal O\Big( \frac{1}{K^2} +\frac{1}{L^4}+ Le^{-cK} +\frac{\epsilon}{\sqrt{KL}}+\epsilon^2+\bar\delta\Big)\,.
\end{equation}
Furthermore, the minimal value of the constant $C$,
given by 
\begin{equation}\label{C1}
C= {B_*^2}(1+\log\frac{A_*}{B_*})^2\,,
\end{equation}
is obtained for the
optimal time-dependent densities, 
\begin{equation}\label{pp'}
\begin{split}
  p'(t,\omega) &=\bar p'(\omega)q'(t) = 
 \left\{
\begin{array}{cl}
 \frac{|\hat f(\omega)|}{\|\hat f\|_{L^1(\rset^d\times [0,t_*])}},\quad    
 & t<t_*:=\min(1,\frac{B_*}{A_*})\,,\\ 
   0,  & t\ge t_*\, ,\\
\end{array}\right. \\
 p(t,\omega) &=\bar p(\omega)q(t) = 
 \left\{
\begin{array}{cl}
 0,    & t<t_*\, ,  \\
   \frac{t^{-1}|\partial_\omega \hat h({F}\omega)|}{\|t^{-1}\partial_\omega\hat h({F}\cdot)\|_{L^1(\rset\times[t_*,1])}},\quad  & t\ge t_*\, ,\\
\end{array}\right. \\
 A_*&:= \|\hat f\|_{L^1(\rset^d)}\,,\\
   B_*&:= F\|\partial_\omega\hat h({F}\cdot)\|_{L^1(\rset)} \exp\big(\| \partial_z\big(zh(z/{F})\big)\|_{L^{\infty}(\rset)}\big)\, ,
\end{split}
\end{equation}
while if $q$ and $q'$ are constant the minimal constant is
\begin{equation}\label{C1*}
\begin{split}
C&=\min\big(2(AB)^{1/2}-B,A\big)\, .
\end{split}
\end{equation}
\end{theorem}

\begin{remark}\label{rem_C}
{\rm
We note that Theorem \ref{thm_main}
shows that functions $f$ with $\|f\|_{L^\infty(\rset^d)}\ll \|\hat f\|_{L^1(\rset^d)}$ can be more accurately approximated by the deep residual network \VIZ{min_barz} as compared to the standard generalization error estimate for shallow
networks \VIZ{min_mean}. Namely,
for the case with time-dependent densities, we have that if $\|f\|_{L^\infty(\rset^d)}$ is sufficiently small compared to $\|\hat f\|_{L^1(\rset^d)}$ then
$\frac{B_*}{A_*} $ becomes small compared to one and we obtain $C \ll A_*^2$ by \VIZ{C1}; in addition the corresponding  error term $\frac{C}{KL}$ in \VIZ{error_estimate} dominates provided
\[
\begin{split}
&L\ll K\ll L^3\, ,\\
&\bar\delta+\epsilon^2\ll (KL)^{-1}\, .
\end{split}
\]
More precisely there holds for a positive constant $c'$ that 
$B_*=\mathcal O(\|f\|_{L^\infty}e^{c'\|f\|_{L^\infty}})$ and $A_*=\|\hat f\|_{L^1}$. For instance a regularized 
discontinuity  as $f(x)= e^{-|x|^2/2} \int_0^{x/a}\frac{\sin t}{t}{\rm d}t $ has
$\|f\|_{L^\infty}\ll \|\hat f\|_{L^1}$ for $0<a\ll 1$. 
Therefore there are functions $f$  such that the deep  residual neural networks \VIZ{min_barz} and \VIZ{min_barze} have a more accurate estimate \VIZ{error_estimate} as compared to the estimate \VIZ{min_mean} for the corresponding neural network with one hidden layer. }
\end{remark}

The optimal densities $\bar p'$ and $\bar p$ that minimize $A$ and $B'$, respectively, for the time-independent densities are
\begin{equation}\label{AB_p}
\begin{split}
    A&= \|\hat f\|^2_{L^1(\rset^d)} \ \mbox{for }\ \bar p'(\omega) = \frac{|\hat f(\omega)|}{\|\hat f\|_{L^1(\rset^d)}}\, ,\\
    B'&=\|F\partial_\omega\hat h(F\omega)\|^2_{L^1(\rset)} \ \mbox{for }\ \bar p(\omega) = \frac{|\partial_\omega \hat h({F}\omega)|}{\|\partial_\omega\hat h({F}\cdot)\|_{L^1(\rset)}}\, ,\\
\end{split}
\end{equation}
as shown in Lemma~\ref{p_opt}, and 
we note when $\frac{B}{A} $ is sufficiently small then $C\simeq 2(AB)^{1/2}\ll A$ by \VIZ{C1*}, so that also deep residual networks with time-independent densities can provide better approximation error as compared to the generalization error for shallow networks \VIZ{min_mean}.
The optimal time-dependent densities are derived in Lemma \ref{MC_quad}. In the particular case when $t_*=1$ in \VIZ{pp'}, the density $q$ concentrates at $t=1$.
%

Analogous to the estimate \VIZ{min_mean} for one hidden layer, we obtain 
by Theorem~\ref{thm_main} a bound on the approximation error for the following minimization problem.
\begin{corollary}
Suppose that the assumptions in Theorem~\ref{thm_main} hold, then
the generalization error has the bound
\[
\min_{%
\substack{
(\omega_{\ell k},\omega'_{\ell k})\in\rset\times\rset^d\\
     \ell=0,\ldots, L-1,\\
     k=1,\ldots, K}
     }%
\mathbb E_{xy}\big[|\bar z_L^*+\GF(x) -y|^2\big]
\le \frac{C}{KL} +\mathcal O\Big( \frac{1}{K^2} +\frac{1}{L^4}+ \bar\delta+ Le^{-cK} +\epsilon^2+\frac{\epsilon}{(KL)^{1/2}}\Big)
\]
where $C$ satisfies \VIZ{C1}.
\end{corollary}

The main idea in the proof is to use a Monte Carlo approximation of the corresponding optimal control problem with an infinite number of layers $L$ and nodes $K$. 
The problem with an infinite number of nodes is solved explicitly in Section \ref{sec_2} and shows that the state, corresponding to $\bar z_\ell$, for each data point $x$ is a linear function in the levels from the initial state to the final state $y(x)$. This simple linear dependence makes it possible to split the Monte Carlo approximation $\RE \sum_{k=1}^{{K}} \bar b_{\ell k}e^{\IU \,  \omega_{\ell k} {\bar z}_\ell} +\RE \sum_{k=1}^{{K}} \bar c_{\ell k}e^{\IU\omega'_{\ell k}\cdot x}$, where the first term approximates a scaled identity map and the second term approximates the data $y(x)$. The variances of these two Monte Carlo approximations are optimized by the sampling densities \VIZ{pp'} and yield smaller error
than the one hidden layer estimate \VIZ{min_mean} in the case $\|f\|_{L^\infty(\rset^d)} \ll \|\hat f\|_{L^1(\rset^d)}$.

To use the approximation result of the theorem in practise requires to
sample roughly from the optimal densities. In Section~\ref{sec_alg} we present 
an explicit layer by layer approximation of \VIZ{min_barz} combined
with an adaptive Metropolis method that approximately samples 
the frequencies optimally, following \cite{ARFM}. 


\section{Random feature Monte Carlo approximation}\label{sec_1-2}
The proof of Theorem \ref{thm_main} is based on Monte Carlo approximation of a solution to an optimal control problem with an infinite number of layers and nodes. The basic case of approximation with one hidden layer is studied in this section.

By setting $\bar b_{\ell k}=0$ in \VIZ{min_barz} or \VIZ{min_barze} we obtain a random Fourier feature network with one hidden layer
\begin{equation*}\label{min_one}
\begin{split}
&\min_{\tiny{\begin{array}{c}
\bar c_{\ell k}\in \cset\\ \ell=1,\ldots, L-1\\ k=1,\ldots,{{K}}
\end{array}}}
\mathbb E_{xy}[|{\bar z}_L-(y-\GF(x))|^2]   
\, ,\\
&\mbox{subject to }\;\;\; {\bar z}_{L} =  
\RE \sum_{\ell=1}^{L-1}\sum_{k=1}^{{K}} \bar c_{\ell k}e^{\IU\omega'_{\ell k}\cdot x}
\, .\\
\end{split}
\end{equation*}
The generalization error \VIZ{one_layer} and the optimal choice of the density, namely $\bar p'=|\hat f|/\|\hat f\|_{L^1(\rset^d)}$ in \VIZ{AB_p}, follows by Monte Carlo approximation of the Fourier representation
\[
f(x)=\int_{\rset^d}\hat f(\omega)e^{\IU\omega\cdot x} \rmD x
\]
using $\bar c_{\ell k}=\hat f(\omega'_{\ell k})/\big(KL\bar p'(\omega'_{\ell k})\big)$ and the following two lemmas. To obtain a corresponding representation for a deep residual network requires more work, using an optimal control problem for infinite number of nodes and layers presented in Section \ref{sec_2} and an error representation in Section \ref{sec_proof}, where an analogous  Monte Carlo approximation is applied.

\begin{lemma}[Monte Carlo quadrature error]\label{MC}Assume that $\omega_j, \, j=1,\ldots, J,$ are independent identically distributed with density $p:\rset^d\to [0,\infty)$. Then the mean, variance and kurtosis for the Monte Carlo approximation $\rJ \sum_{j=1}^J \frac{a(\omega_j)}{p(\omega_j)}$ satisfy
\[
\begin{split}
\mathbb E[\sum_{j=1}^J \frac{a(\omega_j)}{Jp(\omega_j)}] &= \int_{\rset^d} a(\omega) \rmD \omega\, ,\\
\mathbb E[|\sum_{j=1}^J \frac{a(\omega_j)}{Jp(\omega_j)} - \int_{\rset^d} a(\omega) \rmD \omega|^2]&=J^{-1}\Big(
\int_{\rset^d}\frac{|a(\omega)|^2}{p(\omega)} \rmD \omega - \big( \int_{\rset^d} a(\omega) \rmD \omega\big)^2\Big)\, ,\\
\mathbb E[|\sum_{j=1}^J \frac{a(\omega_j)}{Jp(\omega_j)} - \int_{\rset^d} a(\omega) \rmD \omega|^4]&=J^{-2}\Big(
\int_{\rset^d}\frac{|a(\omega)|^2}{p(\omega)} \rmD \omega - \big( \int_{\rset^d} a(\omega) \rmD \omega\big)^2\Big)^2\\
&\quad + J^{-3} \int_{\rset^d}\big(\frac{a(\omega)}{p(\omega)}-\int_{\rset^d} a(\omega') \rmD \omega'\big)^4 p(\omega) \rmD \omega\,.
\end{split}
\]
\end{lemma}
\begin{proof}The proof is well known, and is included here for completeness.
We have  the expected value
\[
\mathbb E_\omega
[\sum_{j=1}^J \frac{a(\omega_j)}{Jp(\omega_j)}] 
= \int_{\rset^{Jd}}\sum_{j=1}^J \frac{a(\omega_j)}{Jp(\omega_j)}\prod_{k=1}^J p(\omega_k) \rmD \omega_k
=\int_{\rset^d} a(\omega) \rmD \omega\, .\\
\]
The variance of this Monte Carlo approximation satisfies
\begin{equation}\label{variance_f}
\begin{split}
&\mathbb E[|\sum_{j=1}^J \frac{a(\omega_j)}{Jp(\omega_j)} - \int_{\rset^d} a(\omega) \rmD \omega|^2]\\
&=\int_{\rset^{Jd}}\big(\sum_{j=1}^J \frac{a(\omega_j)}{Jp(\omega_j)}- \int_{\rset^d} a(\omega) \rmD \omega\big)^2  \prod_{k=1}^J p(\omega_k) \rmD \omega_k\\
&=J^{-2} \int_{\rset^{Jd}}\sum_{j=1}^J \sum_{i=1}^J
\big(\frac{a(\omega_j)}{p(\omega_j)}- \int_{\rset^d} a(\omega) \rmD \omega\big)
\big(\frac{a(\omega_i)}{p(\omega_i)}- \int_{\rset^d} a(\omega) \rmD \omega\big) 
 \prod_{k=1}^Jp(\omega_k) \rmD \omega_k\\
&=J^{-2} \int_{\rset^{Jd}}\sum_{j=1}^J 
\big(\frac{a(\omega_j)}{p(\omega_j)}- \int_{\rset^d} a(\omega) \rmD \omega\big)^2  \prod_{k=1}^Jp(\omega_k) \rmD \omega_k\\
&=J^{-1}\Big(
\int_{\rset^d}\frac{|a(\omega)|^2}{p(\omega)} \rmD \omega - \big( \int_{\rset^d} a(\omega) \rmD \omega\big)^2\Big)\, .
\end{split}
\end{equation}
The estimate of the fourth moment is obtained similarly.
\end{proof}

The following lemma is a classical result in optimal importance sampling.
\begin{lemma}[Optimal importance sampling]\label{p_opt} The optimal probability density 
\begin{equation} \label{eq:optimal_density}
p_*(\omega)= \frac{ |g(\omega)|}{\int_{\rset^d}|g(\omega')| \rmD \omega'}
\,,
\end{equation}
is the solution of the minimization problem
\begin{equation}\label{MC_loss222}
\min_{\substack{p\ge 0,\\ \int_{\R^d} p(\omega)\rmD\omega = 1}} 
\int_{\rset^d}\frac{|g(\omega)|^2}{p(\omega)}
\rmD \omega\,.
\end{equation}
With the choice \VIZ{eq:optimal_density}, for $a(\omega)=g(\omega)$, the Monte Carlo quadrature error becomes
\[
\mathbb E[|\sum_{j=1}^J \frac{a(\omega_j)}{Jp_*(\omega_j)} - \int_{\rset^d} a(\omega) \rmD \omega|^2]=J^{-1}\Big(
(\int_{\rset^d}{|a(\omega)|} \rmD \omega)^2 - \big( \int_{\rset^d} a(\omega) \rmD \omega\big)^2\Big)\, .\\
\]
\end{lemma}
\begin{proof}
The change of variables $p(\omega)=q(\omega)/\int_{\rset^d}q(\omega)\rmD \omega$
implies $\int_{\rset^d}p(\omega) \rmD \omega=1$, for any $q:\rset^d\to[0,\infty)$. We define for any $v:\rset^d\to \rset$ and $\varepsilon$ close to zero
\[
H(\varepsilon):=%
\int_{\rset^d}\frac{|g(\omega)|^2}{q(\omega)+\varepsilon v(\omega)}
\rmD \omega \int_{\rset^d}q(\omega)+\varepsilon v(\omega) \rmD \omega\,.
\]
At the optimum we have
\[
\begin{split}
H'(0)
&= \int_{\rset^d}\frac{|g(\omega)|^2v(\omega)}{-q^2(\omega)}\rmD \omega 
\underbrace{\int_{\rset^d}q(\omega') \rmD \omega' }_{=:c_1}
 + \underbrace{\int_{\rset^d}\frac{|g(\omega')|^2}{q(\omega')} \rmD \omega'}_{=:c_2}
\int_{\rset^d}v(\omega) \rmD \omega\\
&=
\int_{\rset^d}\big(c_2-c_1\frac{|g(\omega)|^2}{q^2(\omega)}\big)v(\omega) \rmD \omega 
\end{split}
\]
and the optimality condition 
$H'(0)=0$ implies
$q(\omega)=\sqrt{\frac{c_1}{c_2}} \, |g(\omega)|
$.
Consequently the optimal density becomes
\begin{equation*} \label{eq:optimal_dens}
p_*(\omega)= \frac{ |g(\omega)|}{\int_{\rset^d}|g(\omega')| \rmD \omega'}
\,.
\end{equation*}
\end{proof}

\section{An optimal control solution}\label{sec_2}
In this section we motivate \VIZ{C1} and \VIZ{pp'} using a solution to an optimal control problem related to \VIZ{min_barz} with an infinite number of nodes  and layers. The complete proof is  presented in Section~\ref{sec_proof}.

In the limit of infinite $L$ and $K$ with 
\[
\bar b_{\ell k}\simeq \frac{\hat b(t_{\ell k}, \omega_{\ell k})}{KL p(t_{\ell k},\omega_{\ell k})}
\mbox{ and } \bar c_{\ell k}\simeq\frac{ \hat c(t'_{\ell k},\omega'_{\ell k})}{KL p'(t'_{\ell k},\omega'_{\ell k})}
\]
 the deep residual neural network problem \VIZ{min_barz}, with the scaling $\delta=\bar\delta$, becomes the optimal control problem
\begin{equation}\label{min_z}
\begin{split}
&\min_{\substack{%
     \hat b:[0,1]\times\rset\to\cset\\
     \hat c:[0,1]\times\rset^d\to \cset
     }}
\mathbb E_{xy}\Big[| z_1-\big(y-\GF(x)\big)|^2 + \delta\int_0^1 |\alpha(t,z_t;x)|^2 \rmD t\Big]
\, ,\\
&\mbox{subject to } \\
& \frac{ \rmD  z_t}{ \rmD t} = \alpha(t,z_t;x)\,,\\
 &z_0 = 0\,.
\end{split}
\end{equation}
We recall the definition of $\alpha(t,z_t;x)$:
\[
\alpha(t,z_t;x)=\RE \int_{\rset} \hat b(t,\omega)e^{\IU\omega\cdot z_t} \rmD \omega 
+\RE \int_{\rset^d} \hat c(t,\omega)e^{\IU\omega\cdot x} \rmD \omega\,.
\]
The motivation for the transformation with $\beta$,  where $z_0=0$ and $z_1\simeq y-\GF(x)$, is that our proof requires in \eqref{beta_zero} $\frac{ \rmD  z_t}{ \rmD t}$ to be small compared to one. By choosing $\GF(x)$ to be a neural network approximation of $f(x)$, i.e., $\GFO(x)$, the difference $y-\GFO(x)$ becomes sufficiently small. 

The noiseless version of \VIZ{min_z} is the optimal control problem 
\begin{equation}\label{alfa}
\begin{split}
&\min_{\alpha:[0,1]\times\rset\times\rset^d\to\rset} \mathbb E_{x}\Big[|\tz_1-\big(f(x)-\GF(x)\big)|^2 +{\delta}\int_0^1 |\alpha(t,\tz_t;x)|^2 \rmD  t\Big]\\
&\mbox{subject to }\\
&\frac{ \rmD \tz_t}{ \rmD t} = \alpha(t,\tz_t;x)\,,\quad t>0\,,\\
&\tz_0=0\,.
\end{split}
\end{equation}
The optimal control problem \VIZ{alfa} can be solved explicitly as described in Lemma~\ref{lemma_z}. This explicit solution is used to derive the bounds \VIZ{C1} and \VIZ{pp'} as follows. The error estimate in Theorem~\ref{thm_main} is based on an estimate of the difference $\bar z_L-\tz_1$, using the optimal
solution, $\tz$, to \VIZ{alfa} with the optimal control written as
\begin{equation}\label{alfa_z}
\alpha(t,\tz_t;x)=\RE \int_{\rset} \hat b(t,\omega)e^{\IU\omega\cdot \tz_t} \rmD \omega 
+\RE \int_{\rset^d} \hat c(t,\omega)e^{\IU\omega\cdot x} \rmD \omega\,, 
\end{equation}
which includes dependence on both $\tz_t$ and $x$. This function $\tz$ is a feasible solution $z$ to \VIZ{min_z}, since it satisfies the differential equation constraint in \VIZ{min_z}. Its relation to the deep residual neural network problems \VIZ{min_barz} and \VIZ{min_barze} is through Monte Carlo quadrature of the integrals, specifically
\begin{equation}\label{MCalfa_z}
\RE \sum_{k=1}^{{K}} \frac{1}{KL}\frac{\hat b_{\ell k}e^{\IU\omega_{\ell k}\cdot \tz_\ell}}{p(t_{\ell k},\omega_{\ell k})}
+\RE \sum_{k=1}^{{K}} \frac{1}{KL} 
\frac{\hat c_{\ell k}e^{\IU\omega'_{\ell k}\cdot x}}{p'(t'_{\ell k},\omega'_{\ell k})}\, 
\, ,
\end{equation}
where $\tz_\ell:=\tz(t_\ell)$ with $t_\ell=Q^{-1}(\frac{\ell}{L})$
and 
\begin{equation}\label{1bc_def}
\hat b_{\ell k}=\hat b(t_{\ell k},\omega_{\ell k}) \mbox{ and } \hat c_{\ell k}=\hat c(t'_{\ell k},\omega'_{\ell k})\,.\end{equation}
Below in Lemma~\ref{MC_quad} we estimate this Monte Carlo quadrature error. The other part of the error for the estimation in Theorem~\ref{thm_main}, including the difference of $\bar z_\ell$ and $\tz_t$, based on the values of $\hat b_{\ell k}$ and $\hat c_{\ell k}$ obtained from \VIZ{alfa_z}, is studied in Section~\ref{sec_proof}.

\begin{lemma}[Optimal pathwise solution]\label{lemma_z}
Given a penalization parameter $\delta>0$, and the optimal control problem
\begin{equation}\label{alfagen}
\begin{split}
&\min_{\CTRL:[0,1]\times\rset\times\rset^d\to\rset} \mathbb E_{x}\Big[|\tz_1-\big(f(x)-\GF(x)\big)|^2 +{\delta}\int_0^1 |\CTRL(t,\tz_t;x)|^2 \rmD  t\Big]\\
&\mbox{subject to }\\
&\frac{ \rmD \tz_t}{ \rmD t} = \CTRL(t,\tz_t;x)\,,\quad t>0\,,\\
&\tz_0=0\,.
\end{split}
\end{equation}
the optimal control solution to \VIZ{alfagen} is given by
\[
\CTRL(t,\tz;x) = \frac{f(x)-\GF(x)}{\delta +1} 
\]
with the optimal path given by a linear function from the 
initial data to the target value
\[
\begin{split}
\tz_t &=  
\frac{t}{1+\delta}\big(f(x)-\GF(x)\big)\, \, ,\\
\frac{ \rmD \tz_t}{ \rmD t} &= \frac{f(x)-\GF(x)}{1+\delta} = \frac{\tz_t}{t}\,,
\end{split}
\]
with the  value
\[
\mathbb E_{x}\Big[| \tz_1-f(x)|^2\Big] 
=\mathbb E_x\Big[|f(x)-\GF(x)|^2\Big]\frac{\delta^2}{(1+\delta )^2}\,.
\]
\end{lemma}

\begin{proof}

While we prove the result for $\delta>0$ we note that the case $\delta=0$ can be treated by taking the limit 
$\delta\rightarrow 0$.
Since the control $\CTRL(t,\tz_t(x);x)$ is allowed to depend on $x$, the minimum in \VIZ{alfagen}  can be performed with respect to each data point $x$ individually, 
hence
\VIZ{alfagen} implies 
\[
\begin{split}
&\min_{\CTRL:[0,1]\times\rset\times\rset^d\to\rset} \Big\{|\tz_1-\big(f(x)-\GF(x)\big)|^2 +{\delta}\int_0^1 |\CTRL(s,\tz_s;x)|^2 \rmD  s\Big\}\\
&\mbox{subject to }\;\;\;\;
\frac{ \rmD \tz_t}{ \rmD t} = \CTRL(t,\tz_t;x)\,,\quad t>0\,,\\
&\tz_0=0\, .
\end{split}
\]
The Lagrangian for this problem becomes
\[
\mathcal L(\tz,\lambda,\CTRL):=  |\tz_1-\big(f(x)-\GF(x)\big)|^2 +{\delta}\int_0^1 |\CTRL_t|^2 \rmD  t + 
\int_0^1 \lambda_t\cdot (\CTRL_t-\frac{ \rmD \tz_t}{ \rmD t} ) \rmD t
\]
with  the corresponding Lagrange multiplier condition
\begin{equation}\label{lagrange_v}
\begin{split}
\frac{ \rmD \lambda_t}{ \rmD t} &= 0\, , \quad t<1\,,\\
\lambda_1 &= 2\Big(\tz_1-\big(f(x)-\GF(x)\big)\Big)\,,
\end{split}
\end{equation}
and the Pontryagin principle
$\CTRL_t = {\argmin}_{a\in\R} (\lambda_t\, a+{\delta}|a|^2)$.
The solution becomes
\[
\CTRL_t = - \frac{1}{\delta}\left(\tz_1-\big(f(x)-\GF(x)\big)\right) \,,
\]
and we obtain
\[
\tz_t =  - \frac{t}{\delta} \left(\tz_1-\big(f(x)-\GF(x)\big)\right) \,,
\]
and thus
\begin{equation}\label{z_1}
\tz_1 = \frac{f(x)-\GF(x)}{1+\delta }\,.
\end{equation}
Therefore we have
\[
\CTRL_t = \frac{f(x)-\GF(x)}{1+\delta} 
\]
and 
\[
\tz_t =  \frac{t}{1+\delta}\big(f(x)-\GF(x)\big)\,,\;\;\;
\frac{ \rmD \tz_t}{ \rmD t} = \frac{\big(f(x)-\GF(x)\big)}{1+\delta} = \frac{\tz_t}{t}\,,
\]
which concludes the proof.

\medskip

\end{proof}

Returning to the problem \eqref{alfa} and substituting 
the particular form of $\alpha$ for $\CTRL$ we have 
that a feasible  solution for \VIZ{min_z} is
\begin{equation}\label{bc_def}
\RE \int_{\rset} \hat b(t,\omega) e^{\IU\omega \tz_t} \rmD \omega + 
\RE \int_{\rset^d} \hat c(t,\omega) e^{\IU\omega\cdot x} \rmD \omega
= \frac{f(x)-\GF(x)}{\delta+1} = \frac{\tz_t}{t}\, .
\end{equation}

Due to the singularity in $\tz_t/t$, at $t=0$ in \VIZ{bc_def}, it is  advantageous, for small $t$, that the function  $\RE\int_{\rset^d} \hat c(t,\omega) e^{\IU\omega\cdot x} \rmD \omega$ approximates $\frac{f(x)-\GF(x)}{\delta+1} $ and $\hat b(t,\cdot)=0$. 
For larger $t$ it is then better to let
\[
\begin{split}
\RE \int_{\rset} \hat b(t,\omega) e^{\IU\omega\cdot \tz_t} \rmD \omega &= \frac{\tz_t}{t}\, .\\
\end{split}
\]
To minimize the Monte Carlo approximation variance, $\mathcal E$, of \VIZ{alfa_z} approximating  the integrals in \VIZ{bc_def} with the sum in \VIZ{MCalfa_z} we can  
choose between these two options using the following construction. We introduce $\gamma:[0,1]\to[0,1]$ and define the convex combination
\[
\RE \int_{\rset} \hat b(t,\omega) e^{\IU\omega \tz_t} \rmD \omega + 
\RE \int_{\rset^d} \hat c(t,\omega) e^{\IU\omega\cdot x} \rmD \omega
= \gamma(t)\frac{f(x)-\GF(x)}{\delta+1} +\big(1-\gamma(t)\big) \frac{\tz_t}{t}\, ,\\
\]
and set
\begin{equation}\label{b_z}
\begin{split}
\RE \int_{\rset} \hat b(t,\omega) e^{\IU\omega\cdot \tz_t} \rmD \omega &= \big(1-\gamma(t)\big)\frac{\tz_t}{t}\, ,\\
\RE \int_{\rset^d} \hat c(t,\omega) e^{\IU\omega\cdot x} \rmD \omega &= \gamma(t)\frac{f(x)-\GF(x)}{\delta+1}
\, ,
\end{split}
\end{equation}
which by Lemma~\ref{MC} implies
\begin{lemma}\label{MC_quad}
Assume $\GF(x)=\RE\sum_{k=1}^{{K}} \frac{\hat f(\omega'_{0 k})e^{\IU\omega'_{0 k}\cdot x}}{K\bar p'(\omega'_{0 k})}$
and let $\hat b_{\ell k}$ and $\hat c_{\ell k}$  be defined by \VIZ{1bc_def} and \VIZ{b_z}, for $k=1,\ldots,{{K}}$ and $\ell=1,\ldots,L-1$, then
\begin{equation}\label{A*B*}
\begin{split}
&\min_{\gamma:[0,1]\to[0,1]}\mathbb E_\omega\Big[\Big|\RE\sum_{\ell=1}^{L-1} \sum_{k=1}^{{K}} \frac{\hat b_{\ell k}e^{\IU\omega_{\ell k} \tz_\ell}}{KLp(t_{\ell k},\omega_{\ell k})}
+\RE \sum_{\ell=1}^{L-1}\sum_{k=1}^{{K}} \frac{\hat c_{\ell k}e^{\IU\omega'_{\ell k}\cdot x}}{KLp'(t_{\ell k},\omega'_{\ell k})}\\
&\quad -\sum_{\ell=1}^{L-1}
\int_{t_\ell}^{t_{\ell+1}}
\big(\RE \int_{\rset} \hat b(t,\omega) e^{\IU\omega \tz_\ell} \rmD \omega
+\RE \int_{\rset^d} \hat c(t,\omega) e^{\IU\omega\cdot x} \rmD \omega\big){\rm dt}\Big|^2\Big]\\
&\le \frac{B_*'^2}{KL}(1+\log\frac{A_*}{B_*'})^2
\end{split}
\end{equation}
for the
optimal time-dependent densities 
\begin{equation}\label{pp'*}
\begin{split}
  p'(t,\omega) &:=\bar p'(\omega)q'(t) = 
 \left\{
\begin{array}{cl}
 \frac{|\hat f(\omega)|}{\|\hat f\|_{L^1(\rset^d\times [0,t_*])}}, \quad    
 & t<t_*:=\min(1,\frac{B_*'}{A_*})\,,\\ 
   0,  & t\ge t_*\, ,\\
\end{array}\right. \\
 p(t,\omega) &:=\bar p(\omega)q(t) = 
 \left\{
\begin{array}{cl}
 0,    & t<t_*\, ,  \\
   \frac{t^{-1}|\partial_\omega \hat h({F}\omega)|}{\|t^{-1}\partial_\omega\hat h({F}\cdot)\|_{L^1(\rset\times[t_*,1])}},\quad  & t\ge t_*\, ,\\
\end{array}\right. \\
 A_*&:= \|\hat f\|_{L^1(\rset^d)}\,,\\
   B_*'&:= F\|\partial_\omega\hat h({F}\cdot)\|_{L^1(\rset)}\, .
\end{split}
\end{equation}
If $q$ and $q'$ are piecewise constant on $[0,1]$, the bound $\frac{B_*'^2}{KL}(1+\log\frac{A_*}{B_*'})^2$  in \VIZ{A*B*} is replaced by 
$$
\min\big(2(AB)^{1/2}-B,A\big)\,,\;\;\;\mbox{using $\gamma(t)=1_{[0,\min(1,\sqrt{B/A}]}(t)$.}
$$
\end{lemma}

\begin{proof}
We note that since $\GF(x)$ is a given neural network function, replacing $c(t,x)$ by $c(t,x)+\GF(x)\frac{\gamma(t)}{1+\delta}$ does not contribute to the approximation error. 
To determine $\hat b$ from $\frac{\tz_t}{t}$
we use that 
the constant $F$ is by definition the bound
\[
{F}=\|f\|_{L^\infty(\rset^d)} \le \int_{\rset^d}|\hat f(\omega)| \rmD \omega\, 
\]
and since by assumption $h(z)=1$ for $|z|\le {1}$,
the map
$z\mapsto z h(z/{F})$ is the identity map, for $|z|\le {F}$, and its Fourier 
transform $\omega\mapsto -\IU\,F\partial_\omega \hat h({F}\omega)$ is a Schwartz function.
The variance of the Monte Carlo error
\[
\begin{split}
\mathcal E(\gamma)&:=
\mathbb E_\omega\big[\big|\RE\sum_{\ell=1}^{L-1} \sum_{k=1}^{{K}} \frac{\hat b_{\ell k}e^{\IU\omega_{\ell k}\cdot \tz_\ell}}{KLp(t_{\ell k},\omega_{\ell k})}
+\RE \sum_{k=1}^{{K'}} \frac{\hat c_{\ell k}e^{\IU\omega'_{\ell k}\cdot x}}{KLp'(t_{\ell k},\omega'_{\ell k})}\\
&\quad -\sum_{\ell=0}^{L-1}
\int_{t_\ell}^{t_{\ell+1}}
\big(\RE \int_{\rset} \hat b(t,\omega) e^{\IU\omega\cdot \tz_\ell} \rmD \omega
+\RE \int_{\rset^d} \hat c(t,\omega) e^{\IU\omega\cdot x} \rmD \omega\big){\rm dt}\big|^2\big]
\, ,
\end{split}\]
has, for optimal densities $p$ and $p'$ by Lemmas  \ref{MC} and \ref{lemma_z}, the bound
\begin{equation}\label{E_var}
\begin{split}
\mathcal E(\gamma)&\le \frac{1}{KL} \big((A_*\int_0^1 \gamma(t)   \rmD  t)^2 + (B_*'\int_0^1\frac{1-\gamma(t)}{t}  \rmD t)^2\big) \\
&\le \frac{1}{KL} \big(A_*\int_0^1 \gamma(t)   \rmD  t + B_*'\int_0^1\frac{1-\gamma(t)}{t}  \rmD t\big)^2
= :\bar{\mathcal E}(\gamma)
\end{split}
\end{equation}
where $\gamma:[0,1]\to [0,1]$.  
 The minimum of $\bar{\mathcal E}(\gamma)$ can be obtained by Euler-Lagrange minimization, namely for any continuous function $v:[0,1]\to \rset$ we have
\[
0=\frac{ \rmD }{ \rmD \mu}\bar{\mathcal E}(\gamma+\mu v)\big|_{\mu=0}
=\frac{2\big(\bar{\mathcal E}(\gamma)\big)^{1/2}}{(KL)^{1/2}} \int_0^1 \Big(A_*  v(t) 
- B_*' \frac{v(t)}{t} \Big) \rmD t
\]
which implies the bang-bang control \VIZ{pp'}, i.e. $\gamma(t)=1_{[0,t_*]}(t)$,
and \[\bar{\mathcal E}(\gamma)=(KL)^{-1}B_*'^2(1+\log\frac{A_*}{B_*'})^2\,.\]

If we assume that $q$ and $q'$ are constant on their support and use a bang-bang control, we obtain 
the minimization problem
\[
\begin{split} 
KL\bar{\mathcal E}(\gamma)\le\min_{\tau\in[0,1]}
(A\int_0^\tau  \rmD t+ \int_\tau^1 \frac{B}{t^2} \rmD t)=&
\min_{\tau\in[0,1]}\big(A\tau+B(\tau^{-1}-1)\big)\\
=&\min\big(2(AB)^{1/2}-B,A\big)
\end{split}
\]
for $\tau=\min(1,(B/A)^{1/2})$ and $\gamma(t)=1_{[0,\tau]}(t)$,
which explains \VIZ{C1*}. 
\end{proof}

It is advantageous for $\alpha$ to first approximate $f(x)-\GF(x)$ and
then after a  positive time  approximate $\tz_t/t$. If $\gamma_t=1$ for all $t$ we have a random Fourier feature network with one hidden layer and $\mathcal E\le (KL)^{-1}A$.

\section{Proof of Theorem \ref{thm_main}}\label{sec_proof}
\begin{proof}[Proof of Theorem \ref{thm_main}]
We denote 
a feasible solution to \VIZ{min_z} by $z(t)$ or $z_t$.
We define, for $\ell=0,\ldots, L$, the time layer values
\[
z_\ell := z(t_\ell)\,,\;\;\;\;\;\mbox{where}\;\;\;
t_\ell:=Q^{-1}\left(\frac{\ell}{L}\right)\,.
\]

To derive an error estimate we choose the discrete amplitudes to be the same as the amplitudes for the
feasible  $z=\tz$ derived in Lemma \ref{lemma_z}, 
using also the optimal partition \VIZ{b_z} with $\gamma(t)=1_{t<\min(1,B_*/A_*)}(t)$, 
for optimal time-dependent densities, and using $\gamma(t)=1_{t<\min(1,(B/A)^{1/2})}(t)$ for densities $q$ and $q'$ that are constant on their support.
Therefore we will use the amplitudes related to \VIZ{MCalfa_z} to define $\bar z$  from $\hat b_{\ell k}=\hat b(t_{\ell k},\omega_{\ell k})$  and
$\hat c_{\ell k}=\hat c(t_{\ell k},\omega_{\ell k})$, which are given in \VIZ{1bc_def} and \VIZ{bc_def} and defined by the solution $\tz=z$.
The aim is to estimate $\mathbb E_{xy}[|\bar z_L^*-\big(y-\GFO(x)\big)|^2]$, where $\bar z^*$ is an optimal solution to \VIZ{min_barz}.
The minimum property of $\bar z_L^*$ implies  
that for any positive ${\zeta}$ we have 
\[
\begin{split}
%
&\mathbb E_{t\omega}\big[\mathbb E_{xy}[|\bar z_L^*-\big(y-\GFO(x)\big)|^2]\big]\\
&\le \mathbb E_{t\omega}\big[\mathbb E_{xy}[|\bar z_L^*-\big(y-\GFO(x)\big)|^2
+\bar\delta L\sum_{\ell=1}^{L-1}|\bar z^*_{\ell+1}-\bar z^*_{\ell}|^2]\big]\\
&\le \mathbb E_{t\omega}\big[\mathbb E_{xy}[|\bar z_L-\big(y-\GFO(x)\big)|^2
+\bar\delta L\sum_{\ell=1}^{L-1}|\bar z_{\ell+1}-\bar z_{\ell}|^2]\big]\\
&=\mathbb E_{t\omega}\big[\mathbb E_{xy}[|(\bar z_L-z_1) +  (z_1-\tz_1)+\tz_1-(y-\GFO(x))|^2]\big] +\mathcal O(\bar\delta)\\
&\le (1+{\zeta}^2)\mathbb E_{t\omega}\big[\mathbb E_{xy}[| \bar z_L-z_1|^2]\big]\\
&\quad +(1+{\zeta}^{-2})\mathbb E_{t\omega}\big[\mathbb E_{xy}[| (z_1-\tz_1)+\tz_1-(y-\GFO(x))|^2]\big] +\mathcal O(\bar\delta)\,.\\
\end{split}
\]
By choosing $z=\tz$, Lemma \ref{lemma_z} yields
\[
\mathbb E_{xy}[| (z_1-\tz_1)+\tz_1-(y-\GFO(x))|^2]
\le 2\mathbb E_{xy}[|\tz_1-\big(f(x)-\GFO(x)\big)|^2]+2\epsilon^2
= \mathcal O(\delta^2 +\epsilon^2)\,, 
\]
and we conclude that
\begin{equation}\label{zeta_use}
\begin{split}
\mathbb E_{t\omega}\big[\mathbb E_{xy}[|\bar z_L^*-\big(y-\GFO(x)\big)|^2]\big]
\le & (1+{\zeta}^2)\mathbb E_{t\omega}\big[\mathbb E_{xy}[| \bar z_L-z_1|^2]\big]\\
&+\mathcal O\big(\bar\delta+(\delta^2+\epsilon^2)(1+\zeta^{-2})\big)\,.
\end{split}
\end{equation}
It remains to estimate the term $\mathbb E_{t\omega}\big[\mathbb E_{xy}[|\bar z_L-z_1|^2]\big]$,
which involves the following  four steps: 
\begin{itemize}
\item[1.] formulate an error representation of the exact path $z_t$ evaluated along the discrete value function for the ordinary difference equation in \VIZ{min_barz},
\item[2.] estimate the derivatives of the value functions for the differential equation in \VIZ{min_z} and the difference equation in \VIZ{min_barz}, 
\item[3.] derive an error estimate for $\mathbb E_{t\omega}[|z_\ell-\bar z_\ell|^2]$, using a discrete Gronwall inequality  applied to the difference of the 
dynamics of $\bar z_\ell$ and $z_\ell$, 
and
\item[4.] use Steps 1,2, and 3 to derive the Monte Carlo quadrature error for 
\[\RE \sum_{\ell=0}^{L-1}\sum_{k=1}^{{K}} \frac{\hat b_{\ell k}e^{\IU\omega_{\ell k}\cdot \bar z_\ell}}{KLp(t_{\ell k},\omega_{\ell k})}
+\RE\sum_{\ell=0}^{L-1}\sum_{k=1}^{{K}} \frac{\hat c_{\ell k}e^{\IU\omega'_{\ell k}\cdot x}}{KLp'(t'_{\ell k},\omega'_{\ell k})}
\, .\] 
\end{itemize}

\noindent{\it Step 1.}[Error representation]  
We define
\[
\begin{split}
\bar B(\bar z_\ell,\ell) & := \RE \sum_{k=1}^{{K}} \frac{\hat b_{\ell k}e^{\IU\omega_{\ell k}\cdot \bar z_\ell}}{KLp(t_{\ell k},\omega_{\ell k})} \\
\bar C(\ell) & := \RE \sum_{k=1}^{{K'}} \frac{\hat c_{\ell k}e^{\IU\omega'_{\ell k}\cdot x}}{KLp'(t'_{\ell k},\omega'_{\ell k})}\\
\bar D(\bar z_\ell,\ell) & := \bar B(\bar z_\ell,\ell) + \bar C(\ell)\,.
\end{split}
\]
We fix a data point $(x,y)$, and then define for this $x$ and for any $z\in\rset$  the discrete value function 
\begin{equation*}
\begin{split}
 \bar u(z,\bar\ell) &=\bar z_L\, ,\\
 \bar z_{\ell +1} &= \bar z_\ell + \bar B(\bar z_\ell,\ell) + \bar C(\ell)\,,\\
\bar z_{\bar\ell} &= z \, .
\end{split}
\end{equation*}
We have
\[
\begin{split}
z_L-\bar z_L&=\bar u(z_L,L)-\bar u(z_0,0)\\
&=\sum_{\ell=0}^{L-1}\big(\bar u(z_{\ell+1},\ell+1)-
\bar u(z_\ell,\ell)\big)
\end{split}
\]
and  by construction there holds for any $z$
\[
\bar u(z,\ell)=\bar u\big(z+\bar D(z,\ell),\ell+1\big)\, .
\]
Introducing for $\ell=0,\ldots,L$  the notation
\[
\begin{split}
B(z_\ell,\ell)&:= \int_{t_\ell}^{t_{\ell+1}}
\int_{\rset} \hat b(t,\omega)e^{\IU\omega\cdot z_t} \rmD \omega \rmD t\, ,\\
 C(\ell)&:=\int_{t_\ell}^{t_{\ell+1}}
\int_{\rset^d} \hat c(t,\omega)e^{\IU\omega\cdot x} \rmD \omega \rmD t\, ,\\
D(z_\ell,\ell) &:= B(z_\ell,\ell)+C(\ell)\, ,\\
\end{split}
\]
telescoping summation implies
\[
\begin{split}
z_L-\bar z_L &= \sum_{\ell=0}^{L-1}
\big(\bar u(z_\ell + D(z_\ell,\ell),\ell+1)-
\bar u(z_\ell+\bar D(z_\ell,\ell),\ell+1)\big)\\
&=\sum_{\ell=0}^{L-1}\int_0^1 \partial_z\bar u\Big(z_\ell+s\big(D(z_\ell,\ell)-\bar D(z_\ell,\ell)\big),\ell+1\Big) \rmD s 
\big(D(z_\ell,\ell)-\bar D(z_\ell,\ell)\big)\\
&=\sum_{\ell=0}^{L-1}\partial_z{\UBB}(z_\ell)
\big(D(z_\ell,\ell)-\bar D(z_\ell,\ell)\big)\,,
\end{split}
\]
where, for the sake of brevity, we defined
\[
\partial_z{\UBB}(z_\ell) = \int_0^1 \partial_z\bar u\Big(z_\ell+s\big(D(z_\ell,\ell)-\bar D(z_\ell,\ell)\big),\ell+1\Big) \rmD s\,.
\]
Next we let
\[
u(z,\ell)=z_L= z_t\big|_{t=1}\, ,
\]
where
\[
\begin{split}
 \frac{ \rmD  z_t}{ \rmD t} &= \RE \int_{\rset} \hat b(t,\omega)e^{\IU\omega\cdot z_t} \rmD \omega 
+\RE \int_{\rset^d} \hat c(t,\omega)e^{\IU\omega'\cdot x} \rmD \omega\, ,\quad t>t_\ell\, ,\\
 z_{t_\ell} &= z\, .
\end{split}
\]
We obtain an error representation based on the exact optimal path $z_\ell$ evaluated along the discrete value function $\bar u$
\begin{equation}\label{1}
\begin{split}
&\mathbb E_{t\omega}[|z_L-\bar z_L|^2] = 
\mathbb E_{t\omega}\Big[\Big(\sum_{\ell=0}^{L-1}\partial_z {\UBB}(z_\ell)\big(D(z_\ell,\ell)-\bar D(z_\ell,\ell)\big)\Big)^2\Big]\\
&=\mathbb E\Big[\Big(\sum_{\ell=0}^{L-1}\partial_z  u(z_\ell,{\ell+1})\big(D(z_\ell,\ell)-\bar D(z_\ell,\ell)\big)\\
&\quad +\sum_{\ell=0}^{L-1}\big(\partial_z {\UBB}(z_\ell)-\partial_z  u(z_\ell,{\ell+1})\big)
\big(D(z_\ell,\ell)-\bar D(z_\ell,\ell)\big)\Big)^2\Big]\\
&\le (1+\zeta^2)\mathbb E_{t\omega}\Big[\Big(\sum_{\ell=0}^{L-1}\partial_z  u(z_\ell,{\ell+1})\big(D(z_\ell,\ell)-\bar D(z_\ell,\ell)\big)\Big)^2\Big]\\
&\quad +(1+\zeta^{-2})\mathbb E_{t\omega}\Big[\Big(
\sum_{\ell=0}^{L-1}\big(\partial_z {\UBB}(z_\ell)-\partial_z  u(z_\ell,{\ell+1})\big)
\big(D(z_\ell,\ell)-\bar D(z_\ell,\ell)\big)\Big)^2\Big]\\
\end{split}
\end{equation}
for any $\zeta>0$.
Introducing
\[
\widetilde{D}(z_\ell,\ell):=
\RE\int_{t_\ell}^{t_{\ell+1}} \int_{\rset}\hat b(t,\omega)e^{\IU\omega\cdot z_\ell} \rmD \omega \rmD t
+\RE\int_{t_\ell}^{t_{\ell+1}} \int_{\rset^d}\hat c(t,\omega)e^{\IU\omega\cdot x} \rmD \omega \rmD t
\]
we have
\begin{equation}\label{Dz}
\begin{split}
D(z_\ell,\ell)-\bar D(z_\ell,\ell) &=
D(z_\ell,\ell)-\widetilde{D}(z_\ell,\ell)
+\widetilde{D}(z_\ell,\ell)- \bar D(z_\ell,\ell)\\
&=\underbrace{\RE\int_{t_\ell}^{t_{\ell+1}} \int_{\rset}\hat b(t,\omega)e^{\IU\omega\cdot z_t} \rmD \omega \rmD t
-\RE\int_{t_\ell}^{t_{\ell+1}} \int_{\rset}\hat b(t,\omega)e^{\IU\omega\cdot z_\ell} \rmD \omega \rmD t}_{=:L^{-1}\xi_a}\\
&\quad +\widetilde{D}(z_\ell,\ell)- \bar D(z_\ell,\ell)\, ,\\
\widetilde{D}(z_\ell,\ell)- \bar D(z_\ell,\ell)&=
\underbrace{\RE\int_{t_\ell}^{t_{\ell+1}} \int_{\rset}\hat b(t,\omega)e^{\IU\omega\cdot z_\ell} \rmD \omega \rmD t
-\RE \sum_{k=1}^{{K}} \frac{\hat b_{\ell k}}{KLp(t_{\ell k},\omega_{\ell k})}
e^{\IU\omega_{\ell k}\cdot z_\ell}}_{=:L^{-1}\xi_b(\ell)}\\
&\  +\underbrace{\RE\int_{t_\ell}^{t_{\ell+1}} \int_{\rset^d}\hat c(t,\omega)e^{\IU\omega\cdot x} \rmD \omega \rmD t
-\RE \sum_{k=1}^{{K'}} \frac{\hat c_{\ell k}}{KLp'(t'_{\ell k},\omega'_{\ell k})}
e^{\IU\omega'_{\ell k}\cdot x}}_{=:L^{-1}\xi_c(\ell)}\,.\\
\end{split}
\end{equation}
The deterministic $\xi_a$ has the 
representation
\begin{equation}\label{xi_a}
L^{-1}\xi_a=\int_{t_\ell}^{t_{\ell+1}} b(t,z_t)  \rmD t - \int_{t_\ell}^{t_{\ell+1}} b(t,z_\ell)  \rmD t
\end{equation}
with the bound $|\xi_a|=\mathcal O(L^{-1})$, using $b(t,z_t)-b(t,z_\ell)=0$ at $t=t_\ell$ and the boundedness of the derivative of $b$. In Step 3, we use $\mathbb E_{t\omega}[\widetilde{D}(z_\ell,\ell)-\bar D(z_\ell,\ell)]=0$ and Lemma \ref{MC}
to establish $\mathbb E_{t\omega}[L^2|\widetilde{D}(z_\ell,\ell)-\bar D(z_\ell,\ell)|^2]=\mathcal O(K^{-1})$.
We shall show in Step 4 that the first term on the right hand side of \VIZ{1} is $\mathcal O(L^{-1}K^{-1})$ by Monte-Carlo quadrature in space-time, using \[\mathbb E_{t\omega}[\sum_{\ell=0}^{L-1}\partial_z  u(z_\ell,{\ell+1})\big(\widetilde{D}(z_\ell,\ell)-\bar D(z_\ell,\ell)\big)]=0\,.\]  Step 3 estimates the second term on the right hand side of \VIZ{1}. The next step estimates the derivatives of $u$ and $\bar u$.

\smallskip

\noindent{\it Step 2.}[Derivatives of the value functions]
The definition $u(z,\ell)=z_1$ and
\[
\begin{split}
\frac{ \rmD  z_s}{ \rmD s} &= \RE \int_{\rset} \hat b(s,\omega)e^{\IU\omega\cdot z_s} \rmD \omega 
+\RE \int_{\rset^d} \hat c(s,\omega)e^{\IU\omega'\cdot x} \rmD \omega\, ,s>t\\
z_t&=z\,,
\end{split}\]
imply that $\partial_z u(z,\ell)=z'_{1,t_\ell}$, where the first variation 
$z'_{s,{t_\ell}}:=\frac{\partial z_s}{\partial  z_{{t_\ell}}} $ solves
\[
\begin{split}
\frac{ \rmD  z'_{s,{t_\ell}}}{ \rmD s} 
&= 
\RE \int_{\rset} \IU\,\hat b(s,\omega)e^{\IU\omega\cdot z_s} \omega z'_{s,{t_\ell}} \, \rmD \omega  
=\partial_zb(s,z_s) z'_{s,{t_\ell}}\,, \quad s>{t_\ell}\, ,\\
z'_{{t_\ell},{t_\ell}} &= 1 \, .
\end{split}
\]
This differential equation has the solution
\begin{equation}\label{uder}
\begin{split}
\partial_z u(z,\ell)&= 
e^{\int_{t_\ell}^1  \partial_z b(s,z_s) \rmD s 
 } \, .\\
\end{split}
\end{equation}
%

To similarly obtain a representation for the gradient of the discrete value function  we use  that $\partial_z\bar u(z,\bar\ell)= \bar z'_{L,\bar\ell} $, where the discrete variation $\bar z'_{\ell,\bar\ell}:=\frac{\partial \bar z_\ell}{\partial \bar z_{\bar\ell}} $ satisfies 
\[
\begin{split}
\bar z'_{\ell+1,\bar\ell} 
&= \big(1 + L^{-1} \bar b'_\ell(\bar z_\ell)\big)\bar z'_{\ell\bar\ell}\, ,\quad
\ell=\bar\ell,\ldots,L-1\, ,\\
\bar z'_{\bar\ell\bar\ell}&=1\, ,
\end{split}
\]
and we write
\[
\big(1 + L^{-1} \bar b'_\ell(\bar z_\ell)\big)\bar z'_{\ell\bar\ell} := \bar z'_{\ell,\bar\ell} 
+ \RE\sum_{k=1}^K
\frac{1}{KL}\IU\frac{\hat b_{\ell k}}{p(t_{\ell k},\omega_{\ell k})} e^{\IU\omega_{\ell k}\cdot \bar z_\ell} \omega_{\ell k}\bar z'_{\ell\bar\ell}\,.
\]
Thus, we have the solution  
\begin{equation}\label{uuder}
\bar z'_{L\bar\ell}= \prod_{\ell=\bar \ell}^{L-1}\big(1 +  L^{-1} \bar b'_\ell(\bar z_\ell)\big)\, .
\end{equation}

We also obtain the second derivative estimate
\begin{equation}\label{Hessian}
|\partial^2_{z} u(z,x,\ell)| \le \int_{t_\ell}^1  
|\partial^2_z b(s,z_s)| e^{\int_{t_\ell}^s|\partial_z b(s,z_r)| \rmD r}
 \rmD s \,\, e^{\int_{t_\ell}^1  |\partial_z b(s,z_s)| \rmD s}\, .
\end{equation}

\smallskip
\noindent{\it Step 3.}[Gronwall inequality] The second term on the right hand side of \VIZ{1} requires an estimate of $z_\ell-\bar z_\ell$. 
Let $\varepsilon_\ell:= z_\ell-\bar z_\ell$ with $\varepsilon_1=z_1-\bar z_1=0$. We first use  a discrete Gronwall inequality, see Lemma \ref{lemma_gronwall}, to estimate $\varepsilon_\ell$ and then estimate $\partial_z{\UBB}-\partial_z u$. 

\smallskip
\noindent{\it Estimate of $\varepsilon_\ell=z_\ell-\bar z_\ell$}. We have by the dynamics \VIZ{min_barz} and \VIZ{min_z}
\begin{equation}\label{Ez}
\begin{split}
\varepsilon_{\ell+1} &=\varepsilon_\ell +\RE \sum_{k=1}^K \frac{\hat b_{\ell k}}{KLp(t_{\ell k},\omega_{\ell k})}
\big(e^{\IU\omega_{\ell k}\cdot z_\ell}-e^{\IU\omega_{\ell k}\cdot \bar z_\ell}\big)\\
&\quad +
\underbrace{\RE\int_{t_\ell}^{t_{\ell+1}} \int_{\rset}\hat b(t,\omega)e^{\IU\omega\cdot z_t} \rmD \omega \rmD t
-\RE\int_{t_\ell}^{t_{\ell+1}} \int_{\rset}\hat b(t,\omega)e^{\IU\omega\cdot z_\ell} \rmD \omega \rmD t}_{=:L^{-1}\xi_a}\\
&\quad +
\underbrace{\RE\int_{t_\ell}^{t_{\ell+1}} \int_{\rset}\hat b(t,\omega)e^{\IU\omega\cdot z_\ell} \rmD \omega \rmD t
-\RE \sum_{k=1}^K \frac{\hat b_{\ell k}}{KLp(t_{\ell k},\omega_{\ell k})}
e^{\IU\omega_{\ell k}\cdot z_\ell}}_{=:L^{-1}\xi_b(\ell)}\\
&\quad +\underbrace{\RE\int_{t_\ell}^{t_{\ell+1}} \int_{\rset^d}\hat c(t,\omega)e^{\IU\omega\cdot x} \rmD \omega \rmD t
-\RE \sum_{k=1}^K \frac{\hat c_{\ell k}}{KLp'(t'_{\ell k},\omega'_{\ell k})}
e^{\IU\omega'_{\ell k}\cdot x}}_{=:L^{-1}\xi_c(\ell)}\\
&=: \varepsilon_\ell +\RE \sum_{k=1}^K \frac{\hat b_{\ell k}}{KLp(t_{\ell k},\omega_{\ell k})}
e^{\IU\omega_{\ell k}\cdot z_\ell}
\big(1-e^{-\IU\omega_{\ell k}\cdot \varepsilon_\ell}\big)
+L^{-1}(\xi_a+\xi_b+\xi_b)\, .
\end{split}
\end{equation}
The term $L^{-1}\xi_a$ is a deterministic quadrature problem \VIZ{xi_a} and since
$ \rmD b(t,z_t)/ \rmD t$ is bounded we have $|\xi_a|=\mathcal O(L^{-1})$. 
Using the Monte Carlo quadrature estimate in Lemma \ref{MC} we obtain for $j=b,c$
\begin{equation}\label{2}
\begin{split}
\xi_a &=\mathcal O(L^{-1})\, ,\\
\mathbb E_{t\omega}[\xi_j] &=0\, ,\\
\mathbb E_{t\omega}[|\xi_j|^2] &=\mathcal O(K^{-1})\, ,\\
\mathbb E_{t\omega}[|\xi_j|^4] &=\mathcal O(K^{-2})\, .\\
\end{split}
\end{equation}

The next step is to estimate the sum over $k$ on the right hand side of \VIZ{Ez}. We have
\[
|1-e^{-\IU\omega_{\ell k}\cdot\varepsilon_\ell}|\le |\omega_{\ell k}| |\varepsilon_\ell |
\]
which implies 
\[
a_\ell:=\frac{|\RE  \sum_{k=1}^{{K}} \frac{\hat b_{\ell k}}{Kp(t_{\ell k},\omega_{\ell k})}
e^{\IU\omega_{\ell k}\cdot z_\ell}
\big(1-e^{-\IU\omega_{\ell k}\cdot \varepsilon_\ell}\big)|}{|\varepsilon_\ell|}
\le  \sum_{k=1}^{{K}} \frac{|\hat b_{\ell k}||\omega_{\ell k}|}{Kp(t_{\ell k},\omega_{\ell k})}\,.
\]
Next we determine a probability that $a_\ell$ are bounded by a number $\alpha$ for  $\ell=0,\ldots,L$. 
By the  Chernoff bound in Lemma \ref{lemma_chenrnoff}, based on
$\frac{|\hat b(t_{\ell k},\omega_{\ell k})|}{p(t_{\ell k},\omega_{\ell k})}|\omega_{\ell k}|$ being bounded, we have
\[
|a_\ell|\le \sum_{k=1}^K \frac{|\hat b_{\ell k}|}{Kp(t_{\ell k},\omega_{\ell k})}|\omega_{\ell k}| = 
\int_{\rset}|\hat b(t_\ell,\omega)||\omega| \rmD \omega + K^{-1}\nu_\ell
\]
where $\nu_\ell$ is a random variable where ${\mathbb{P}}( \frac{|\nu_\ell |}{K^{}}>\alpha)$ is bounded by
$e^{-cK\alpha^2}$ for some positive constant $c$. We need the probability for the event $\max_\ell \frac{|\nu_\ell |}{K^{}}>\alpha$
which then by independence becomes 
\[
\mathbb P(\max_\ell \frac{|\nu_\ell |}{K^{}}>\alpha)\le Le^{-cK\alpha^2}
\]
for some positive constant $c$. 

For any $\zeta$, in particular  $\zeta=L^{-1/2}$, we obtain
\begin{equation}\label{3}
\begin{split}
|\varepsilon_{\ell +1}|^2 &= |\varepsilon_\ell + \RE \sum_{k=1}^K \frac{\hat b_{\ell k}}{KLp(t_{\ell k},\omega_{\ell k})}
e^{\IU\omega_{\ell k}\cdot z_\ell}
\big(1-e^{-\IU\omega_{\ell k}\cdot \varepsilon_\ell}\big)
+L^{-1}(\xi_a+\xi_b+\xi_c)|^2\\
&\le |\varepsilon_\ell |^2(1 + L^{-1}a_\ell)^2(1+\zeta^2) + (1+\zeta^{-2})L^{-2} |\xi_a+\xi_b+\xi_c|^2\\
&\le |\varepsilon_\ell |^2(1 + L^{-1}a_\ell)^2(1+L^{-1}) + (1+L)L^{-2} |\xi_a+\xi_b+\xi_c|^2\,.
\end{split}
\end{equation}
When all $a_\ell$ are bounded by $\alpha$, we can apply Gronwall's inequality in Lemma \ref{lemma_gronwall} to
obtain $|\varepsilon_\ell|^2 = \mathcal O(\sum_{\ell=0}^{L-1} |\xi_a(\ell)+\xi_b(\ell)+\xi_c(\ell)|^2 L^{-1})$.
In the complement set, namely the event $\max_\ell \frac{|\nu_\ell |}{K^{}}>\alpha$, which has a small probability bounded by $Le^{-cK\alpha^2}$, the differences  $|\varepsilon_\ell|$ are uniformly bounded, since by assumptions the data $|f(x)|$ are bounded in the maximum norm, and
$\|\frac{| \partial_\omega \hat h({F}\cdot)|^2}{p(\cdot)}\|_{L^{\infty}(\rset)}+
\|\frac{\hat f}{p'}\|_{L^\infty(\rset^d)}$ are bounded, which implies that the maximum norm of $|\bar z|$ is uniformly bounded with respect to all variables.
Consequently there is a positive constant $c$ such that  $\mathbb E_{t\omega}[|\varepsilon_\ell|^2] =\mathcal O (L^{-2}+K^{-1} + Le^{-cK})$, for all $\ell$. 

Squaring \VIZ{3}
yields the recursion
\begin{equation}\label{4}
\begin{split}
|\varepsilon_{\ell +1}|^4 & \le
|\varepsilon_\ell |^4(1 + L^{-1}a_\ell)^4(1+L^{-1})^3 + (1+L)^3L^{-4} |\xi_a+\xi_b+\xi_c|^4
\end{split}
\end{equation}
%
%
and a similar Gronwall estimate implies  that there is a positive constant $c$ such that
\begin{equation}\label{eps_est}
\begin{split}
\mathbb E_{t\omega}[|\varepsilon_\ell|^2] &=\mathcal O (L^{-2}+K^{-1} + Le^{-cK})\, ,\\
\mathbb E_{t\omega}[|\varepsilon_\ell|^4] &=\mathcal O(L^{-4}+K^{-2}+Le^{-cK})\, .\\
\end{split}
\end{equation}

\smallskip
\noindent{\it Estimate of $\partial_z{\UBB}-\partial_z u$.}
{To estimate $\partial_z{\UBB}-\partial_z u$} we perform the splitting 
\begin{equation}\label{uu_diff}
\begin{split}
&\partial_z{\UBB}(z_\ell)-\partial_z u(z_\ell,{\ell+1})\\
&=\int_0^1 \partial_z\bar u\Big(z_\ell+s\big(D(z_\ell,\ell)-\bar D(z_\ell,\ell)\big),\ell+1\Big) \rmD s
-\partial_z u(z_\ell,{\ell+1})\\
&=
\int_0^1\Big[\partial_z\bar u\Big(z_\ell+s\big(D(z_\ell,\ell)-\bar D(z_\ell,\ell)\big),\ell+1\Big)\\
&\quad -\partial_z u\Big(z_\ell+s\big(D(z_\ell,\ell)-\bar D(z_\ell,\ell)\big),{\ell+1}\Big)\Big] \rmD s\\
&\qquad +\int_0^1[\partial_zu\Big(z_\ell+s\big(D(z_\ell,\ell)-\bar D(z_\ell,\ell)\big),{\ell+1}\Big)-\partial_zu(z_\ell,{\ell+1}) ] \rmD s\,.
\end{split}
\end{equation}
The second term has the estimate
\begin{equation}\label{uu_diff2}
\begin{split}
&|\int_0^1[\partial_zu\Big(z_\ell+s\big(D(z_\ell,\ell)-\bar D(z_\ell,\ell)\big),{\ell+1}\Big)-\partial_zu(z_\ell,{\ell+1}) ] \rmD s|\\
&\le   \|\partial^2_z u\|_{L^\infty}|D(z_\ell,\ell)-\bar D(z_\ell,\ell)|
\end{split}
\end{equation}
where we use \VIZ{Hessian} to bound the second derivatives of $u$.
To bound the first term on the right hand side
we use for each $s\in[0,1]$ the path $z(t;s)$ that solves the differential equation \VIZ{min_z}, with the given $\hat b(t,\omega)$
and $\hat c(t,\omega)$ obtained from $\tz$, and   at time  
$t_{\ell+1}$ the path starts with value $z_\ell + s\big(D(z_\ell,\ell)-\bar D(z_\ell,\ell)\big)$. Then we estimate the difference $\partial_z\bar u-\partial_z u$ along the same such paths $z(t;s)$. At all time levels $t_{\ell'}$, 
$\ell'=\ell,\ldots,L$ we still denote the path value $z_{\ell'}$. 
We have by \VIZ{uder} 
\[
\partial_z u(z_\ell,{\ell+1})=e^{\int_{t_{\ell+1}}^1 b(s,z_s) \rmD s}
\] 
and with probability  $1-Le^{-cK\alpha^2}$ by \VIZ{uuder} and \VIZ{Ez}
\[
\begin{split}
\partial_z {\UBB}(z_\ell)&= \prod_{\ell'=\ell}^{L-1}
(1+\underbrace{\RE\sum_{k=1}^K
\frac{\hat b_{\ell' k}e^{\IU\omega_{\ell'k}\cdot \bar z_{\ell'}}}{KLp(t_{\ell' k},\omega_{\ell'k})}}_{=:\hat b_{\ell'}})\\
&= \prod_{\ell'=\ell}^{L-1}
e^{\hat b_{\ell'}+\mathcal O(|\hat b_{\ell'}|^2) }\\
&=e^{\sum_{\ell'=\ell}^{L-1}\big(\hat b_{\ell'}+\mathcal O(|\hat b_{\ell'}|^2)\big)}\\
&=\exp\big({\mathcal O(L^{-1})+\sum_{\ell'=\ell}^{L-1}\RE\sum_{k=1}^K
\frac{\hat b_{\ell' k}e^{\IU\omega_{\ell'k}\cdot \bar z_{\ell'}}}{KLp(t_{\ell' k},\omega_{\ell'k})}}\big)\, .\\
\end{split}
\]
As in \VIZ{Ez} we obtain
\[
\begin{split}
&\RE\sum_{\ell'=\ell}^{L-1}\sum_{k=1}^K
\frac{\hat b_{\ell' k}e^{\IU\omega_{\ell'k}\cdot \bar z_{\ell'}}}{KLp(t_{\ell' k},\omega_{\ell'k})}\\ &=
\RE\sum_{\ell'=\ell}^{L-1}\sum_{k=1}^K
\frac{\hat b_{\ell' k}e^{\IU\omega_{\ell'k}\cdot  z_{\ell'}}}{KLp(t_{\ell' k},\omega_{\ell'k})}
-\RE\sum_{\ell'=\ell}^{L-1}\sum_{k=1}^K
\frac{\hat b_{\ell' k}e^{\IU\omega_{\ell'k}\cdot  z_{\ell'}}}{KLp(t_{\ell' k},\omega_{\ell'k})}(1-e^{-\IU\omega_{\ell'k}\cdot\varepsilon_{\ell'}})\\
&= \sum_{\ell'=\ell}^{L-1}\Big(
\int_{t_{\ell}}^{t_{\ell+1}}b(t,z_{\ell'}) \rmD t +\frac{\xi_b(\ell')}{L}+\RE\sum_{k=1}^K
\frac{\hat b_{\ell' k}e^{\IU\omega_{\ell'k}\cdot  z_{\ell'}}}{KLp(t_{\ell' k},\omega_{\ell'k})}(1-e^{-\IU\omega_{\ell'k}\cdot\varepsilon_{\ell'}})\Big)\\
\end{split}
\]
so that with probability $1-Le^{-cK\alpha^2}$
\[
\begin{split}
|\partial_z {\UBB}(z_\ell)-\partial_z u(z_\ell,\ell+1)|^2
= 
\mathcal O(L^{-2}) + \mathcal O(\sum_{\ell=\ell'}^{L-1}\frac{| \varepsilon_{\ell'}|^2+|\xi_b(\ell')|^2}{L}) 
\end{split}
\]
and by \VIZ{eps_est}, \VIZ{uu_diff} and \VIZ{uu_diff2}
\begin{equation}\label{ud}
\begin{split}
\mathbb E_{t\omega}[|\partial_z {\UBB}(z_\ell) -\partial_z u(z_\ell,\ell+1)|^2]=\mathcal O(K^{-1}+L^{-2}+Le^{-cK})\,,\\
\mathbb E_{t\omega}[|\partial_z {\UBB}(z_\ell) -\partial_z u(z_\ell,\ell+1)|^4]=\mathcal O(K^{-2}+L^{-4}+Le^{-cK})\,,\\
\mathbb E_{t\omega}[|D(z_\ell,\ell)-\bar D(z_\ell,\ell)|^2]=L^{-2}\mathcal O(K^{-1}+L^{-2})\,,\\
\mathbb E_{t\omega}[|D(z_\ell,\ell)-\bar D(z_\ell,\ell)|^4]=L^{-4}\mathcal O(K^{-2}+L^{-4})\,.\\
\end{split}
\end{equation}

We are now ready to estimate the second term on the right hand side of \VIZ{1}, namely
\begin{equation*}
\begin{split}
&(1+\zeta^{-2})\mathbb E_{t\omega}\Big[\Big(
\sum_{\ell=0}^{L-1}\big({\partial_z} {\UBB}(z_\ell)-{\partial_z}  u(z_\ell,{\ell+1})\big)
\big(D(z_\ell,\ell)-\bar D(z_\ell,\ell)\big)\Big)^2\Big]\,.\\
\end{split}
\end{equation*}
Introducing the notation $\Delta u_\ell:={\partial_z} {\UBB}(z_\ell)-{\partial_z}  u(z_\ell,{\ell+1})$
and  $\Delta D_\ell:= D(z_\ell,\ell)-\bar D(z_\ell,\ell)$  
we have
\begin{equation}\label{2-4}
\begin{split}
    &\mathbb E_{t\omega}\Big[\Big(
\sum_{\ell=0}^{L-1} \big({\partial_z} {\UBB}(z_\ell)-{\partial_z}  u(z_\ell,{\ell+1})\big)
\big( D(z_\ell,\ell)-\bar D(z_\ell,\ell)\big)\Big)^2\Big]\\
&= \sum_{\ell=0}^{L-1}\sum_{\ell'=0}^{L-1}
\mathbb E_{t\omega}\big[(\Delta u_\ell\Delta D_\ell)
(\Delta u_{\ell'}\Delta D_{\ell'})\big]\\
&\le  \sum_{\ell=0}^{L-1}\sum_{\ell'=0}^{L-1}
\big(\mathbb E_{t\omega}[|\Delta u_\ell|^2|\Delta u_{\ell'}|^2]\,
\mathbb E_{t\omega}[|\Delta D_\ell|^2|\Delta D_{\ell'}|^2]\big)^{1/2}\\
&=L^2\mathcal O(K^{-1}+L^{-2}+Le^{-cK})\mathcal O(K^{-1}+L^{-2})L^{-2}
=\mathcal O(K^{-2}+L^{-4}+Le^{-cK})\, ,
\end{split}
\end{equation}
using \VIZ{ud}. The next step estimates the first term on the right hand side of \VIZ{1}.

\smallskip
\noindent{\it Step 4.}[Monte Carlo quadrature error] We will see that provided
$L\ll K\ll L^3$, the generalization error is dominated by
\begin{equation*}
\begin{split}
R_0&:= \mathbb E_{t\omega}\Big[\Big(\sum_{\ell=0}^{L-1} {\partial_z} u(z_\ell,\ell+1) \big(
D(z_\ell)-\bar D(z_\ell)\big)\Big)^2\Big]
=\mathbb E_{t\omega}\Big[\Big(\sum_{\ell=0}^{L-1} {\partial_z} u(z_\ell,\ell+1)\times\\ &\qquad\times \big\{
\RE\int_{t_{\ell}}^{t_{\ell+1}} \int_{\rset} \hat b(t,\omega)e^{\IU\omega\cdot z_t} 
 \rmD \omega \rmD  t 
-\RE\sum_{k=1}^K \frac{\hat b_{\ell k}}{KLp(t_{\ell k},\omega_{\ell k})} e^{\IU\omega_{\ell k}\cdot z_\ell }
\\ &\qquad\quad
+ \RE\int_{t_{\ell}}^{t_{\ell+1}} \int_{\rset^d} \hat c(t,\omega)e^{\IU\omega\cdot x} 
 \rmD \omega \rmD  t 
-\RE\sum_{k=1}^K \frac{\hat c_{\ell k}}{KLp'(t'_{\ell k},\omega'_{\ell k})} e^{\IU\omega'_{\ell k}\cdot x}\big\}
\Big)^2\Big]\, .\\
\end{split}
\end{equation*}
The sum over $\ell $ in $R_0$ 
can for any $\zeta\in\rset$ be split (to separate out the mean zero terms) as in \VIZ{Dz} 
\begin{equation}\label{R0_comp}
\begin{split}
&R_0 = \mathbb E_{t\omega}\Big[\Big(
\sum_{\ell=0}^{L-1} {\partial_z} u(z_\ell,\ell+1) \big(D(z_\ell)-\widetilde{D}(z_\ell)\big)
+ \sum_{\ell=0}^{L-1}{\partial_z} u(z_\ell,\ell+1) \big(
\widetilde{D}(z_\ell)-\bar D(z_\ell)\big)\Big)^2\Big]\\
&\le (1+\zeta^2)\mathbb E_{t\omega}\Big[\Big( \sum_{\ell=0}^{L-1}{\partial_z} u(z_\ell,\ell+1) \big(
\widetilde{D}(z_\ell)-\bar D(z_\ell)\big)\Big)^2\Big]\\
&\quad +
(1+\zeta^{-2}) \mathbb E_{t\omega}\Big[\Big(
\sum_{\ell=0}^{L-1} {\partial_z} u(z_\ell,\ell+1) \big(D(z_\ell)-\widetilde{D}(z_\ell)\big)\Big)^2\Big]\\
&=(1+\zeta^2)\mathbb E_{t\omega}\Big[\Big(\sum_{\ell=0}^{L-1} {\partial_z} u(z_\ell,\ell+1)\times\\
&\qquad\times\big\{
\RE\int_{t_{\ell}}^{t_{\ell+1}} \int_{\rset} \hat b(t,\omega)e^{\IU\omega\cdot z_\ell} 
 \rmD \omega \rmD  t 
-\RE\sum_{k=1}^K \frac{\hat b_{\ell k}}{KLp(t_{\ell k},\omega_{\ell k})} e^{\IU\omega_{\ell k}\cdot z_\ell }
\\ &\qquad\quad
+ \RE\int_{t_{\ell}}^{t_{\ell+1}} \int_{\rset^d} \hat c(t,\omega)e^{\IU\omega\cdot x} 
 \rmD \omega \rmD  t 
-\RE\sum_{k=1}^K \frac{\hat c_{\ell k}}{KLp'(t'_{\ell k},\omega'_{\ell k})} e^{\IU\omega'_{\ell k}\cdot x}\big\}
\Big)^2\Big]\\
&\quad +(1+\zeta^{-2})\mathbb E_{t\omega}\Big[\Big(\sum_{\ell=0}^{L-1} {\partial_z} u(z_\ell,\ell+1) \times\\
&\qquad\times
\big\{
\RE\int_{t_{\ell}}^{t_{\ell+1}} \int_{\rset} \hat b(t,\omega)e^{\IU\omega\cdot z_t} 
 \rmD \omega \rmD  t -
\RE\int_{t_{\ell}}^{t_{\ell+1}} \int_{\rset} \hat b(t,\omega)e^{\IU\omega\cdot z_\ell} 
 \rmD \omega \rmD  t \big\}\Big)^2\Big]
\, .\\
\end{split}
\end{equation}
To estimate 
the second expected value on the right hand side
we will use that $\frac{ \rmD  z_t}{ \rmD t}$ is small in the following sense. 
Lemma \ref{lemma_z}  implies that
\[
\frac{ \rmD  z_t}{ \rmD t}=\frac{f(x)- \GF(x)
}{1+\delta}
\] 
which is sufficiently small in order to obtain
\begin{equation}\label{beta_zero}
\begin{split}
&\mathbb E_{t\omega}\Big[\Big(\sum_{\ell=0}^{L-1} {\partial_z} u(z_\ell,\ell+1)
\big\{
\RE\int_{t_{\ell}}^{t_{\ell+1}} \int_{\rset} \hat b(t,\omega)e^{\IU\omega'_{\ell k}\cdot z_t} 
 \rmD \omega \rmD  t -
\RE\int_{t_{\ell}}^{t_{\ell+1}} \int_{\rset} \hat b(t,\omega)e^{\IU\omega\cdot z_\ell} 
 \rmD \omega \rmD  t \big\}\Big)^2\Big]\\
&\le 
\mathbb E_{t\omega}\Big[\Big(\sum_{\ell=0}^{L-1} |{\partial_z} u(z_\ell,\ell+1)| \int_{t_{\ell}}^{t_{\ell+1}} \int_{\rset} |\hat b(t,\omega)||\omega||z_t-z_\ell|
 \rmD \omega \rmD  t \big)\Big)^2\Big]\\
&\le \mathbb E_{t\omega}\Big[\Big(\sum_{\ell=0}^{L-1} |{\partial_z} u(z_\ell,\ell+1)|\sup_t\int_{\rset}\hat b(t,\omega)||\omega| \rmD \omega \max(t_{\ell+1}-t_\ell)^2||\frac{f(x)-\GF(x)
}{1+\delta}|
\Big)^2\Big]\\
&=\mathcal O(L^{-2}K^{-1})\, . 
\end{split}
\end{equation}
With $\GF=0$ we would instead obtain the bound $\mathcal O(L^{-2})$.

To estimate the first expected value on the right hand side of \VIZ{R0_comp} we note that we have 
\[
\mathbb E_{t\omega}[\RE\sum_{k=1}^K\frac{\hat c_{\ell k}}{KLp'(t'_{\ell k},\omega'_{\ell k})}e^{\IU\omega_{\ell k}'\cdot x}] =
\RE\int_{t_{\ell}}^{t_{\ell+1}}\hat c(t,\omega) e^{\IU\omega\cdot x} \rmD \omega \rmD  t
\]
and by Lemma \ref{MC}
\[
\mathbb E_{t\omega}\big[\big(\RE\sum_{k=1}^K\frac{\hat c_{\ell k}}{KLp'(t'_{\ell k},\omega'_{\ell k})}e^{\IU\omega_{\ell k}'\cdot x} -
\RE\int_{t_{\ell}}^{t_{\ell+1}}\hat c(t,\omega) e^{\IU\omega\cdot x} \rmD \omega \rmD  t\big)^2\big]
=\mathcal O(K^{-1}L^{-2})\, .
\]
All terms 
\[
\RE\sum_{k=1}^K\frac{\hat c_{\ell k}}{KLp'(t'_{\ell k},\omega'_{\ell k})}e^{\IU\omega_{\ell k}'\cdot x} -
\RE\int_{t_{\ell}}^{t_{\ell+1}}\hat c(t,\omega) e^{\IU\omega\cdot x} \rmD \omega \rmD  t
\] 
are independent
for different $\ell$, and ${\partial_z} u(z_\ell)$ is deterministic. 
The zero mean and independence properties of the terms in the first expected value on the right hand side of \VIZ{R0_comp} imply 
\[
\begin{split}
&\mathbb E_{t\omega}\Big[\Big(\sum_{\ell=0}^{L-1} {\partial_z} u(z_\ell,\ell+1)
\big\{
\RE\int_{t_{\ell}}^{t_{\ell+1}} \int_{\rset} \hat b(t,\omega)e^{\IU\omega\cdot z_\ell} 
 \rmD \omega \rmD  t 
-\RE\sum_{k=1}^K \frac{\hat b_{\ell k}}{KLp(t_{\ell k},\omega_{\ell k})} e^{\IU\omega_{\ell k}\cdot z_\ell }
\\ 
&\qquad\quad
+ \RE\int_{t_{\ell}}^{t_{\ell+1}} \int_{\rset^d} \hat c(t,\omega)e^{\IU\omega\cdot x} 
 \rmD \omega \rmD  t 
-\RE\sum_{k=1}^K \frac{\hat c_{\ell k}}{KLp'(t'_{\ell k},\omega'_{\ell k})} e^{\IU\omega'_{\ell k}\cdot x}\big\}
\Big)^2\Big]\\
&=\sum_{\ell=0}^{L-1}\mathbb E_{t\omega}\Big[\Big( \partial_z u(z_\ell,\ell+1)
\big\{
\RE\int_{t_{\ell}}^{t_{\ell+1}} \int_{\rset} \hat b(t,\omega)e^{\IU\omega\cdot z_\ell} 
 \rmD \omega \rmD  t 
- \RE\sum_{k=1}^K \frac{\hat b_{\ell k}}{KLp(t_{\ell k},\omega_{\ell k})} e^{\IU\omega_{\ell k}\cdot z_\ell }\big\}\Big)^2
\\ 
&\quad
+ \Big( \partial_z u(z_\ell,\ell+1)
\big\{\RE\int_{t_{\ell}}^{t_{\ell+1}} \int_{\rset^d} \hat c(t,\omega)e^{\IU\omega\cdot x} 
 \rmD \omega \rmD  t 
-\RE\sum_{k=1}^K \frac{\hat c_{\ell k}}{KLp'(t'_{\ell k},\omega'_{\ell k})} e^{\IU\omega'_{\ell k}\cdot x}\big\}
\Big)^2\Big]\,.
\end{split}
\]


The variance of these Monte Carlo approximations yields as in Lemma \ref{MC} and Lemma \ref{MC_quad}
\[
R_0= \mathcal O(L^{-1}K^{-1} + L^{-4})\, ,
\]
where the constant in $\mathcal O(L^{-1}K^{-1})$ is
\[
\int_0^1\int_{\rset} \frac{|\hat b(t,\omega)|^2|{\partial_z} u(z_t,t)|^2}{p(t,\omega)}  \rmD \omega \rmD t
+\int_0^1\int_{\rset^d} \frac{|\hat c(t,\omega)|^2|{\partial_z} u(z_t,t)|^2}{p'(t,\omega)}  \rmD \omega \rmD t\, ,
\]
which by Lemma \ref{p_opt} becomes minimal for
\[
p(t,\omega)= \frac{|\hat b(t,\omega)||{\partial_z} u(z_t,t)|}{\int_0^1\int_{\rset} |\hat b(t,\omega)||{\partial_z} u(z_t,t)|
  \rmD \omega \rmD t}
\]
and 
\[
p'(t,\omega)= \frac{|\hat c(t,\omega)||{\partial_z} u(z_t,t)|}{\int_0^1\int_{\rset^d} |\hat c(t,\omega)||{\partial_z} u(z_t,t)|
  \rmD \omega \rmD t}\, .
\]
Here we use the notation $u(z,t):=z_1$ where $z_t=z$ and $z_s$ solves the differential equation constraint in \VIZ{min_z} for $s>t$.

For a deep residual neural network to be as accurate as a residual neural network with one hidden layer, asymptotically as $KL\to\infty$, we therefore need
\[
\frac{1}{KL} + \frac{1}{K^2}+\frac{1}{L^4}\simeq \frac{1}{KL}
\]
which implies
\[
L\ll K\ll L^3\, ,
\]
and by choosing $\zeta=\big((\delta^2+\epsilon^2)KL\big)^{1/2}$ in \VIZ{zeta_use} we obtain \VIZ{error_estimate}.
\end{proof}

\begin{lemma}[Gronwall's inequality]\label{lemma_gronwall}
Assume that there are positive constants $\Gamma$ and $\Lambda$ such that
\[
\eta_{\ell}\le \Gamma \eta_{\ell-1} +\Lambda\,,\quad \ell=1,2,\ldots,L\,.
\]
Then 
\[
\eta_\ell\le \left\{\begin{array}{cc}
\Gamma^\ell \eta_0 + \Lambda \frac{\Gamma^\ell-1}{\Gamma-1}\,, & \Gamma\ne 1\,,\\
\eta_0 + \Lambda\ell\,, & \Gamma=1\,.
\end{array}
\right.
\]
\end{lemma}
\begin{proof} The proof is well known and included here for completeness.
We have 
\begin{equation}\label{lemma2_est}
\begin{split}
\eta_\ell &\le \Gamma \eta_{\ell-1} +\Lambda\\
& \le \Gamma(\Gamma \eta_{\ell-2}+\Lambda) + \Lambda\\
& = \Gamma^2 \eta_{\ell-2} + \Lambda(1+\Gamma)\\
& \le \Gamma^2(\Gamma \eta_{\ell-3} +\Lambda) + \Lambda(1+\Gamma)\\
& \le \Gamma^3 \eta_{\ell-3} + \Lambda(1+\Gamma+\Gamma^2)\\
& \le \Gamma^\ell \eta_0 + \Lambda(1+\Gamma+\ldots + \Gamma^{\ell-1})\\
& = \left\{\begin{array}{cc}
\Gamma^\ell \eta_0 + \Lambda \frac{\Gamma^\ell-1}{\Gamma-1}\,, & \Gamma\ne 1\,,\\
\eta_0 + \Lambda\ell\,, & \Gamma=1.
\end{array}
\right.
\end{split}
\end{equation}
\end{proof}

\begin{lemma}[Chernoff bound]\label{lemma_chenrnoff}
Assume that $\xi_j\,, j=1,\ldots,J$ are independent identically distributed random variables on $\rset$ with mean zero, $\mathbb E[\xi_1]=0$, and the exponential moment bound $\mathbb E[e^{\theta\xi_1}]<\infty$, for $\theta$ in a neighborhood of zero. Then the probability for large deviations of the 
empirical mean  has the exponential bound
\begin{equation}\label{chernoff_bound}
\mathbb P(|\sum_{j=1}^J \frac{\xi_j}{J}| \ge \alpha)\le e^{-JI(\alpha)}+e^{-JI(-\alpha)}\,,
\end{equation}
where the Legendre transform $I(\alpha):=\sup_{\theta\in\rset} (\theta\alpha-\log(\mathbb E[e^{\theta \xi_1}])$ is positive for $\alpha \in\rset\setminus\{0\}$.
\end{lemma}
\begin{proof}
The proof is well known and included here for completeness.
Taylor expansion implies 
\[
I(\alpha)\ge \theta\alpha-\frac{\theta^2\mathbb E[\xi_1^2]}{2} +o(\theta^2)
\]
so that $I(\alpha)>0$ for $\alpha\ne 0$.
We have for $\theta\ge 0$
\begin{equation}\label{chern_proof}
\begin{split}
    \mathbb P\Big(\frac{1}{J}\sum_{j=1}^J \xi_j \ge \alpha\Big) 
    &=\mathbb E[1_{\{\sum_{j=1}^J {\xi_j} \ge J\alpha\}}]\\
    &\le \mathbb E[e^{-\theta J\alpha +\theta \sum_{j=1}^J {\xi_j}}
    1_{\{\sum_{j=1}^J {\xi_j} \ge J\alpha\}}]\\
     &\le e^{-J\theta \alpha} \mathbb E[ e^{\theta \sum_{j=1}^J {\xi_j}}]\\
     &= e^{-J\theta\alpha} (\mathbb E[ e^{\theta \xi_1}])^J\\
     &=e^{-J(\theta\alpha-\log \mathbb E[ e^{\theta \xi_1}])}\,.\\
\end{split}
\end{equation}
The estimate  $\mathbb P(\frac{1}{J}\sum_{j=1}^J \xi_j \le -\alpha)\le e^{-J(\theta\alpha-\log \mathbb E[ e^{-\theta \xi_1}])}$ follows by applying \VIZ{chern_proof} to $-\xi_j$
and  we obtain  \VIZ{chernoff_bound} by maximizing with respect to $\theta\in\rset$.
\end{proof}

\section{Numerical algorithms}\label{sec_alg}
This section presents numerical algorithms based on the deep residual network optimization problem \VIZ{min_barz}. 
First, we motivate an explicit layer by layer method to approximate the deep residual network optimization problem \VIZ{min_barz} for given frequency distributions. This layer by layer method is then used in combination with an adaptive Metropolis random Fourier feature method to sample the frequencies in an optimal way, without explicit knowledge of the frequency distribution. The sampled frequencies with their corresponding amplitudes are then used as initial values to train the final residual network with a global optimizer.

\subsection{A layer by layer approximation of amplitudes}
Consider the problem with infinite number of layers and nodes in Lemma \ref{lemma_z}
and replace $f(x)$ by $y(x)$ and set $\GF(x)=0$.  As shown in \VIZ{lagrange_v}, the Lagrange multiplier \[
\lambda_t(x)=2\big(\tz_1(x)-y(x)\big)\,,
\]
is constant for each data point $x$ and layer $t$.  We write the control as
\[
\begin{split}
\alpha(t,\tz_t(x),x)&=\RE \int_{\rset} \hat b(t,\omega)e^{\IU\omega\cdot \tz_t(x)} \rmD \omega 
+\RE \int_{\rset^d} \hat c(t,\omega)e^{\IU\omega'\cdot x} \rmD \omega\\
&=\RE\big(S_b(x)\hat b(t,\cdot) +S_c(x)\hat c(t,\cdot)\big)
\end{split}
\]
and let 
\[
\begin{split}
a(t,\cdot)&:=\left[
\begin{array}{c}
\hat b(t,\cdot)\\
\hat c(t,\cdot)
\end{array} \right]\,,\\
S(x)&:=[S_b(x)\ \ S_c(x)]\,,
\end{split}
\]
so that $\alpha=\RE (Sa)$. By
the Pontryagin principle, the problem to determine  the amplitudes $\hat b$ and $\hat c$  becomes
\begin{equation}\label{lambda_delta}
\begin{split}
&\argmin_{\hat b(t,\cdot),\hat c(t,\cdot)}\mathbb E_{xy}[\lambda_t(x) \RE\big(S(x)a\big)+\delta \Big(\RE\big(S(x)a\big)\Big)^2]\\
&=\argmin_{\hat b(t,\cdot),\hat c(t,\cdot)}\mathbb E_{xy}[\frac{\lambda_t(x)}{\delta} \RE\big(S(x)a\big)+ \Big(\RE\big(S(x)a\big)\Big)^2]\\
&=\argmin_{\hat b(t,\cdot),\hat c(t,\cdot)}\mathbb E_{xy}[|\frac{\lambda_t(x)}{2\delta} + \RE\big(S(x)a\big)|^2]\,. \\
\end{split}
\end{equation}
Based on a finite set of data $\{(x_n,y(x_n))\}_{n=1}^N$, 
the number of levels $L$ and nodes $K$,
and with the purpose to motivate the layer by layer approximation,
we  assume that the residual is reduced by  a factor $\delta$, from the levels $L$ to $\ell$ as
\[
\frac{\lambda_{t_\ell}(x)}{2\delta}=\frac{2( z_L(x)-y(x)\big)}{2\delta}\approx \bar z_\ell(x)-y(x)\,.
\] 
%
%
Then  by \VIZ{lambda_delta}, we obtain the
explicit layer by layer least squares approximation
\begin{equation}\label{layer}
\begin{split}
(\bar b_{\ell \cdot},\bar c_{\ell \cdot})&=\argmin_{{\small\begin{array}{c}
\bar b_{\ell k}\in\cset\\
\bar c_{\ell k}\in\cset
\end{array}}}\sum_{n=1}^N |\bar z_\ell(x_n)
+ {\rm Re}\big(\sum_{k=1}^K \bar b_{\ell k}e^{\IU\omega_{\ell k}\bar z_\ell(x_n)}+
\sum_{k=1}^K \bar c_{\ell k}e^{\IU\omega'_{\ell k}\cdot x_n}\big)- y(x_n)|^2\,,\\
\bar z_{\ell+1}(x)&=\bar z_\ell(x) + \RE\big(\sum_{k=1}^K \bar b_{\ell k}e^{\IU\omega_{\ell k}\bar z_\ell(x)}+
\sum_{k=1}^K \bar c_{\ell k}e^{\IU\omega'_{\ell k}\cdot x}\big)\,,\quad \ell=0,\ldots,L-1\,,\\
\bar z_0&=0\,.
\end{split}
\end{equation}

\subsection{Layer by layer training of amplitudes and frequencies}
In this subsection we present Algorithm \ref{alg:lblARFM} which generalizes \VIZ{layer} to also adaptively sample the frequencies. The frequencies are updated by the Adaptive Metropolis sampling Algorithm \ref{alg:ARFM}, which approximately samples the optimal frequencies. 

Algorithm \ref{alg:ARFM} has been introduced and described in detail in \cite{ARFM}. Algorithm \ref{alg:ARFM} is based on the computed absolute values of the amplitudes, aimed to become equidistributed by Metropolis updates of the frequencies. To have the Metroplis test related to equidistributed amplitudes, i.e. amplitudes with the same absolute values, lead to approximate optimal sampling of frequencies, as described  in \cite{ARFM}. From 
a given set of training data $\{(x_n, y_n)\}_{n=1}^N$ the amplitudes are determined by solving linear least squares problems for the different frequencies updates
\begin{equation}\label{eq:num_disc_prob}
    \min_{\boldsymbol{\hat{\beta}}\in \mathbb{C}^K}\big(N^{-1}|\BARS\boldsymbol{\hat{\beta}}-\mathbf y|^2 + \hat \delta|\boldsymbol{\hat{\beta}}|^2\big)\,,
\end{equation}
where $\BARS\in\mathbb{C}^{N\times K}$ is the matrix with elements $\BARS_{n,k} = e^{{\IU}\omega_k\cdot x_n}$, $n = 1,\ldots,N$, $k = 1,\ldots,K$, $\mathbf y=(y_1,\ldots,y_N)$ $\in \mathbb{R}^N$ and $\hat\delta$ is a non negative Tikhonov parameter.


When training layers deeper than the first hidden layer we use Algorithm \ref{alg:ARFM} with a slight modification. In the layer by layer method \VIZ{layer} the problem defining the amplitudes is changed to the following least squares problem for the training data $\{(x_n, y_n)\}_{n=1}^N$,
\begin{equation}\label{eq:disc_lsq}
    \min_{\boldsymbol{\bar c}, \boldsymbol{\bar b}\in \mathbb{C}^{K}}\big(N^{-1}|\BARS
    \left[
\begin{array}{c}
\boldsymbol{\bar c}\\
\boldsymbol{\bar b}
\end{array} \right]\\
    -\mathbf r|^2 + \hat\delta|(\boldsymbol{\bar c}, \boldsymbol{\bar b})|^2\big)\,,
\end{equation}
where $\BARS\in\mathbb{C}^{N\times 2K}$ is the matrix with elements $\BARS_{n,k} = e^{{\IU}\omega'_{\ell k}\cdot x_n}$ for $n = 1,\ldots,N$, $k = 1,\ldots,K$, and $\BARS_{n,k} = e^{{\IU}\omega_{\ell k} {\bar z}_\ell(x_n)}$ for $n = 1,\ldots,N$, $k = K+1,\ldots,2K$ and $\mathbf r=(y_1-{\bar z}_\ell(x_1),\ldots,y_N-{\bar z}_\ell(x_N))\in \mathbb{R}^N$. The frequencies $\omega_{\ell k}$ are samples from the standard normal distribution 
and are never updated. 

Algorithm \ref{alg:lblARFM} is based  on the layer by layer method \VIZ{layer} with each 
least squares solution replaced by calling the adaptive Metropolis Algorithm \ref{alg:ARFM}.
When running Algorithm \ref{alg:ARFM} with the modification  \VIZ{eq:disc_lsq} used in Algorithm \ref{alg:lblARFM} we say that we run the \emph{modified} Algorithm \ref{alg:ARFM}.
In the layer by layer method \VIZ{layer} the first layer starting with $\bar z_0=0$ will precisely do the initial step of constructing $\beta$ in \VIZ{min_ls}, \VIZ{betastar} and \VIZ{min_barz}. Therefore Algorithms \ref{alg:lblARFM} and \ref{alg:prepost} are formulated without mentioning  $\beta$. 

\subsection{Global post-training}
In the layer by layer build up of a residual neural network each layer is optimized one at a time. As post-training after running Algorithm \ref{alg:lblARFM} a global optimizer, training the whole residual network \VIZ{min_barz}, is used to further increase the approximation property of the residual neural network as presented in Algorithm \ref{alg:prepost}. In Algorithm \ref{alg:prepost} the  number of training points $N_1$ for the pre-training can be chosen significantly smaller than $N$ according to our experimental results.

\begin{algorithm}[tb]
\caption{Adaptive random Fourier features with Metropolis sampling}\label{alg:ARFM}
\begin{algorithmic}
\STATE {\bfseries Input:} $\{(x_n, y_n)\}_{n=1}^N$\COMMENT{data}
\STATE {\bfseries Output:} $x\mapsto\sum_{k=1}^K\hat\beta_ke^{{\IU}\omega_k\cdot x}$\COMMENT{random features}
\STATE Choose a sampling time $T$, a proposal step length $\check\delta$, an exponent $\gamma$, a Tikhonov parameter $\hat\delta$ and a frequency $m$ of  $\boldsymbol{\hat{\beta}}$ updates
\STATE $M \gets \mbox{ integer part}\, (T/\check\delta^2)$
\STATE ${\BFO} \gets \textit{the zero vector in $\rset^{Kd}$}$
\STATE $\boldsymbol{\hat{\beta}} \gets \textit{minimizer of  problem \VIZ{eq:num_disc_prob} given } \BFO$
\FOR{$i = 1$ {\bfseries to} $M$}
    \STATE $r_{\mathcal{N}} \gets \textit{standard normal random vector in $\rset^{Kd}$}$
    \STATE $\BFO' \gets \BFO + \check\delta r_{\mathcal{N}}$ \COMMENT{random walk Metropolis proposal}
    \STATE $\boldsymbol{\hat{\beta}}' \gets \textit{minimizer of  problem \VIZ{eq:num_disc_prob} given } \BFO'$
    \FOR{$k = 1$ {\bfseries to} $K$}
        \STATE  $r_{\mathcal{U}} \gets \textit{sample from uniform distribution on $[0,1]$}$ %
        \IF {$|\hat{\beta}'_k|^\ALPHAEX/|\hat{\beta}_k|^\ALPHAEX>r_{\mathcal{U}}$\COMMENT{Metropolis test}}
                    \STATE $\omega_{k} \gets \omega'_k$
                    \STATE $\hat{\beta}_{k} \gets \hat{\beta}'_k$
                \ENDIF
            \ENDFOR
            \IF {$i \mod m = 0$}
                \STATE $\boldsymbol{\hat{\beta}} \gets \textit{minimizer of problem \VIZ{eq:num_disc_prob} with adaptive } \BFO$
            \ENDIF
\ENDFOR
\STATE $\boldsymbol{\hat{\beta}} \gets \textit{minimizer of problem \VIZ{eq:num_disc_prob} with adaptive } \BFO$
\STATE $x\mapsto\sum_{k=1}^K\hat\beta_ke^{{\IU}\omega_k\cdot x}$

\end{algorithmic}
\end{algorithm}

\begin{algorithm}[tb]
\caption{Layer by layer adaptive random Fourier features with Metropolis sampling}\label{alg:lblARFM}
\begin{algorithmic}
\STATE {\bfseries Input:} $\{(x_n, y_n)\}_{n=1}^N$\COMMENT{data}
\STATE {\bfseries Output:} $x\mapsto \bar{z}_L(x)$\COMMENT{Residual network}
\STATE Choose the number of layers $L$
\STATE ${\bar z}_1 \gets \textit{real part of output from Algorithm \ref{alg:ARFM} run on data } \{(x_n, y_n)\}_{n=1}^N$
\STATE $\{r_n\}_{n=1}^N \gets \{y_n-{\bar z}_1(x_n)\}_{n=1}^N$
\FOR{$\ell = 2$ {\bfseries to} $L$}
\STATE  $Sa\gets \textit{real part of output from modified Algorithm \ref{alg:ARFM} run on data } \{(x_n, r_n)\}_{n=1}^N$
\STATE  $\bar z_{\ell}=\bar z_{\ell-1}+Sa$
    \STATE $\{r_n\}_{n=1}^N \gets \{y_n-{\bar z}_{\ell}(x_n)\}_{n=1}^N$
\ENDFOR
\STATE $x\mapsto \bar{z}_L(x)$

\end{algorithmic}
\end{algorithm}



\begin{algorithm}[tb]
\caption{Global optimization of layer by layer pre-trained residual network}\label{alg:prepost}
\begin{algorithmic}
\STATE {\bfseries Input:} $\{(x_n, y_n)\}_{n=1}^N$\COMMENT{data}
\STATE {\bfseries Output:} $x\mapsto \bar{z}_L(x)$\COMMENT{Residual network}
\STATE Choose the number of layers $L$ and the number of data points $N_1\leq N$ for the pre-training
\STATE $\bar{z}_L\gets \textit{output from Algorithm \ref{alg:lblARFM} run on data } \{(x_n, y_n)\}_{n=1}^{N_1}$
\STATE $\bar{z}_L\gets \textit{output from global optimizer,  based on \VIZ{min_barz}, run on } \bar{z}_L\textit{ and data } \{(x_n, y_n)\}_{n=1}^{N}$
\STATE $x\mapsto \bar{z}_L(x)$

\end{algorithmic}
\end{algorithm}



\begin{figure}[t!]
\centering
\begin{minipage}{0.45\textwidth}
\includegraphics[width=0.98\textwidth]{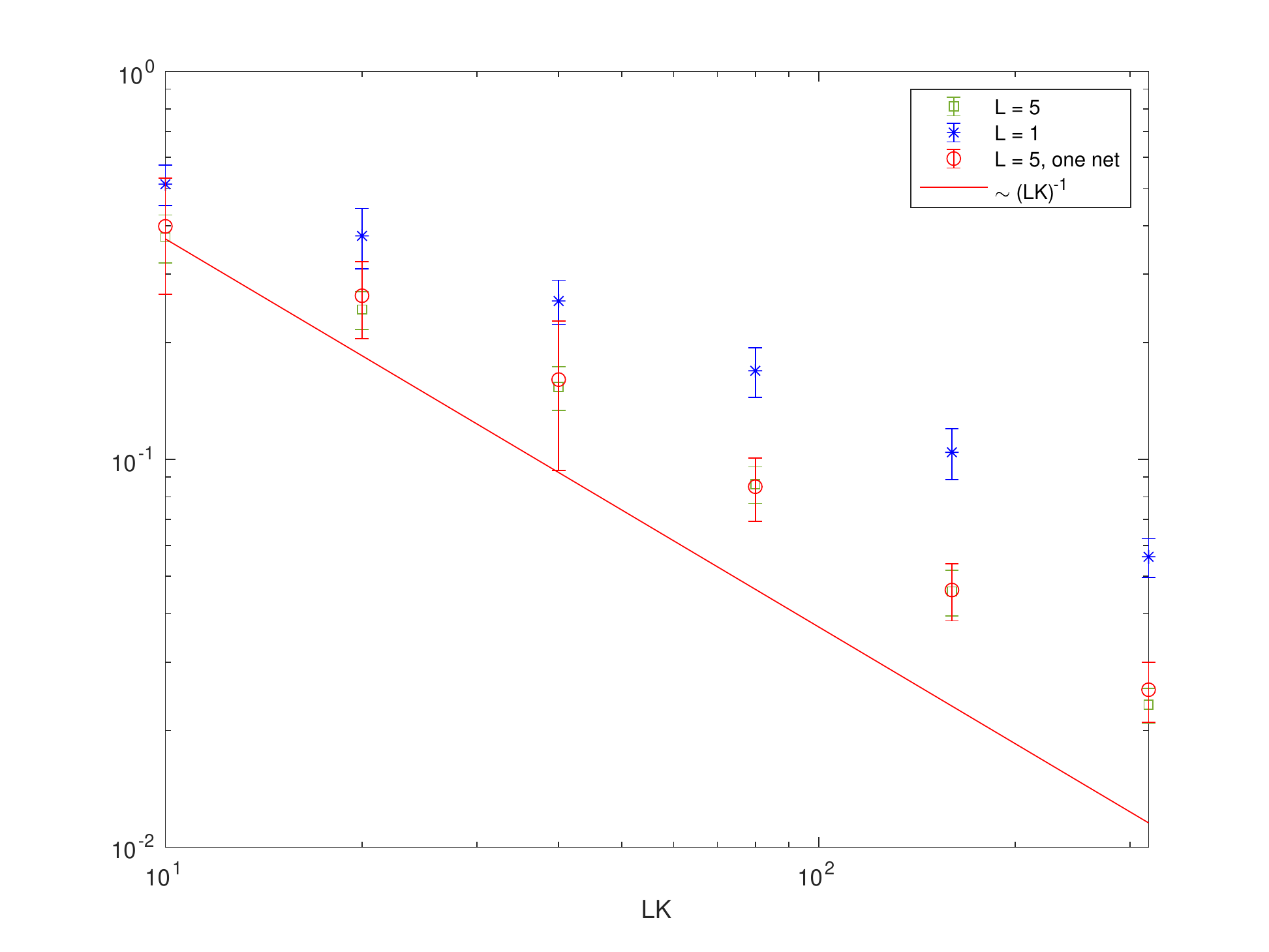}
\caption{Method 1, dimension $d=3$: dependence of the generalization error on the total number of nodes, $KL$, with the error bars defined by \VIZ{eq:error_bars}. Presented errors: a one layer neural network, 
a five layer neural network, and a five layer neural network for which $\bar{z}_{\ell+1} - \bar{z}_{\ell}$ is defined to be a one layer neural network with frequencies in dimension $d + 1$ and data for $L > 1$ are of the form $\{(x_n, \bar{z}_{\ell}(x_n)), y_n\}^{N}_{n=1}$.}
\label{fig:case1_LK}
\end{minipage}
\quad
\begin{minipage}{0.45\textwidth}
\includegraphics[width=0.98\textwidth]{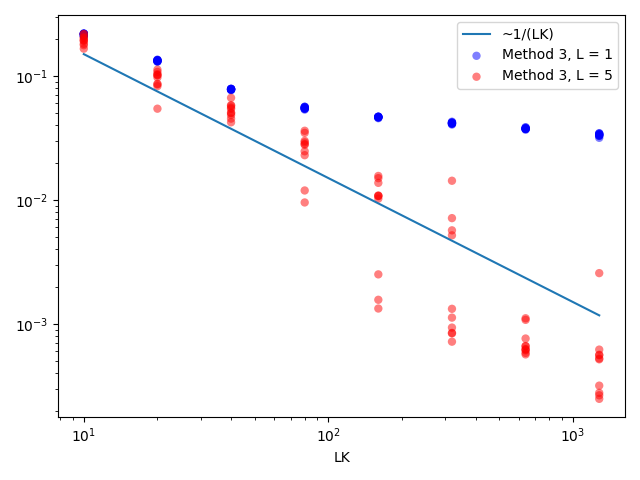}
\caption{Method 3, dimension $d=4$: dependence of the generalization error on  the total number of nodes, $KL$. For each value of $KL$ the $11$ blue dots and the $10$ red dots represent the different outcomes of the generalization error for $L = 1$ and $L = 5$, respectively. The blue dots are almost on top of each other.
}\label{fig:case3_LK}
\end{minipage}
\end{figure}

\section{Numerical experiments}\label{sec:numerics}
In this section we present numerical experiments. Our first objective is to study how the generalization error behaves in practice with respect to the total number of nodes $KL$. The second objective is to study how the generalization error depends on the choice of algorithm to train the residual neural network.  This is performed  using three methods:
\begin{itemize}
    \item Method 1: The residual networks are trained by Algorithm \ref{alg:lblARFM}.
    \item Method 2: The residual networks are trained by a global optimizer.
    \item Method 3: The residual networks are trained by Algorithm \ref{alg:prepost}.
\end{itemize}
Consider for $x\in\rset^d$ the target functions \[f_1(x) = \mathrm{Si}\left(\frac{x_1}{a}\right)e^{-\frac{|x|^2}{2}}\] and \[f_2(x) = \mathrm{Si}\left(\frac{x_1}{a}\right)e^{-\frac{|x_1|^2}{2}}\] where $a = 10^{-2}$ and \[\mathrm{Si}(v) := \int_0^v \frac{\sin(t)}{t}\mathrm{d}t\] is the so called \emph{Sine integral}. We generate training data by sampling the components of $x_n = (x_{n1},$ $x_{n2}, \ldots, x_{nd})\in\rset^d$ from the standard normal distribution, setting $y_n = f_1(x_n)$ or $y_n = f_2(x_n)$ and normalizing the data by subtracting the mean and dividing by the standard deviation component wise. The test data is generated analogously but normalized with the mean and the standard deviation computed from the training data set. The parameter choices are presented in Table \ref{table:paramsum}. For Method 1, error bars are generated and presented in Figure \ref{fig:case1_LK}. The error bars are defined as the closed intervals 
\begin{equation}\label{eq:error_bars}
    [e_{K} - 2{\sigma}_K, e_{K} + 2{\sigma}_K]
\end{equation}
where ${\sigma}_K$ is the empirical standard deviation of the generalization error and $e_{K}$ is the empirical mean of the generalization errors after $\bar M$ independent realizations.

When running Method 3 we generate Figure \ref{fig:case3_LK} where instead of error bars we present $11$ outcomes for each $KL$ when $L = 1$ and $10$ outcomes for each $KL$ when $L = 5$.

\subsection{Method 1: The residual networks are trained by Algorithm \ref{alg:lblARFM}}
With Method 1 we build up the residual neural network layer by layer with Algorithm \ref{alg:lblARFM}. In Algorithm \ref{alg:ARFM},  the number of iterations per layer is chosen large with the purpose of minimizing the error. Figure \ref{fig:case1_LK} presents the generalization error with respect to $KL$ for two different values of the number of layers $L$.

\subsection{Method 2: The residual networks are trained by a global optimizer}
In Method 2 we run the global optimizer Adam, see \cite{kingma2014adam}, with frequencies and amplitudes initialized by Xavier normal initialization, see \cite{Xavier}, with a learning rate decreasing as $\Delta t_e / t_e$ where $\Delta t_e$ is the initial learning rate and $t_e$ is the epoch number $t_e$. The fact that the choice of initial distribution of the frequencies and amplitudes is not obvious in all situations is a partial motivation for Algorithm \ref{alg:prepost}, which with Algorithm \ref{alg:lblARFM} asymptotically samples the optimal frequencies and solves for the corresponding amplitudes layer by layer and uses these values as initial values for the global optimizer.

We present the validation and training mean squared errors, with respect to the number of epochs running Adam in Figure \ref{fig:L5_dec_dt}.

In Table \ref{table:stat_errs} we present the empirical average of the generalization errors and standard deviation for $\bar{M}$ runs of Method 2 for the target functions $f_1$ and $f_2$ and in Figures \ref{fig:target_fcn_and_nn_f1} and \ref{fig:target_fcn_and_nn_f2} we present the target functions and neural networks plotted for one run each.

\subsection{The residual networks are trained by Algorithm \ref{alg:prepost}}
With Method 3 we run Algorithm \ref{alg:prepost}. That is, we build up a residual neural network layer by layer which we use as an initial neural network for a global optimizer. In this case we choose Adam with decreasing learning rate as the global optimizer.

The main objective of Method 3 is to study how the generalization error depends on $KL$,
see Figure \ref{fig:case3_LK}. 
Figures \ref{fig:target_fcn_and_nn_f1} and \ref{fig:target_fcn_and_nn_f2} show plots of the neural network and Table \ref{table:stat_errs} presents the empirical average of the generalization errors and the standard deviation after $\bar{M}$ independent runs.




\begin{table}[t!]
\caption{The generalization error from $\bar{M}$ independent runs for data in dimension $d = 10$.}
{
\small
\centering
\bgroup
\def\arraystretch{1.4}
\begin{tabular}{|l|c|c|c|c|}
\hline
Target function                                                                                      & \multicolumn{2}{c|}{$f_1$}                                                                                                                  & \multicolumn{2}{c|}{$f_2$}                                                                                                                  \\ \hline
Method                                                                                               & \begin{tabular}[c]{@{}c@{}}$2$\\ Xavier\\ \&\\ Adam\end{tabular} & \begin{tabular}[c]{@{}c@{}}$3$\\ Layer by layer\\ \&\\ Adam\end{tabular} & \begin{tabular}[c]{@{}c@{}}$2$\\ Xavier\\ \&\\ Adam\end{tabular} & \begin{tabular}[c]{@{}c@{}}$3$\\ Layer by layer\\ \&\\ Adam\end{tabular} \\ \hline
$\bar{M}$                                                                                            & \multicolumn{2}{c|}{$10$}                                                                                                                   & $20$                                                             & $10$                                                                     \\ \hline
\begin{tabular}[c]{@{}l@{}}Empirical average\\ of the generalization\\ error\end{tabular}            & $0.0448$                                                         & $0.0011$                                                                 & $2.69\times 10^{-4}$                                             & $1.05 \times 10^{-5}$                                                    \\ \hline
\begin{tabular}[c]{@{}l@{}}Empirical standard\\ deviation of the\\ generalization error\end{tabular} & $0.0209$                                                         & $0.0005$                                                                 & $2.34\times 10^{-4}$                                             & $3.94\times 10^{-6}$                                                     \\ \hline
\end{tabular}
\egroup
}
\label{table:stat_errs}
\end{table}

\begin{figure}[t!]
\centering
\begin{minipage}{0.45\textwidth}
\includegraphics[width=0.95\linewidth]{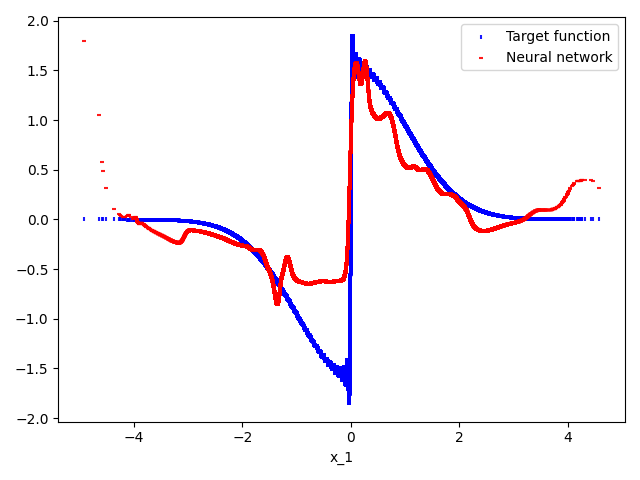}
\caption{Method 2: Xavier \& Adam}
\label{fig:tfnn_c2s}
\end{minipage}
\quad
\begin{minipage}{0.45\textwidth}
\includegraphics[width=0.95\linewidth]{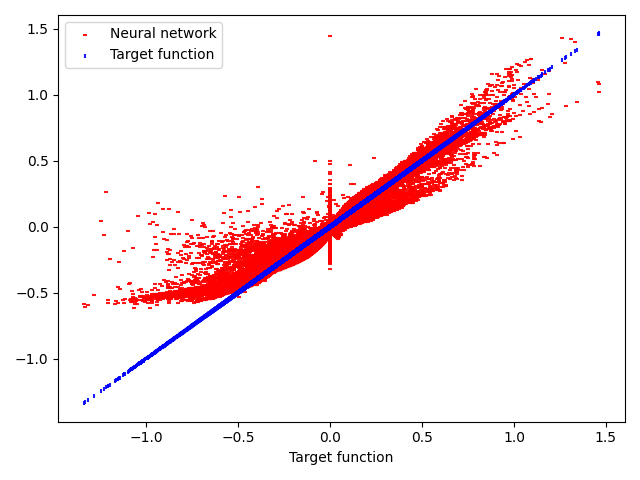}
\caption{Method 2: Xavier \& Adam. \newline
Correlation between neural network and target function: $0.968$.}
\label{fig:tfnn_c2f}
\end{minipage}

\vspace{0.4cm}

\begin{minipage}{0.45\textwidth}
\includegraphics[width=0.95\linewidth]{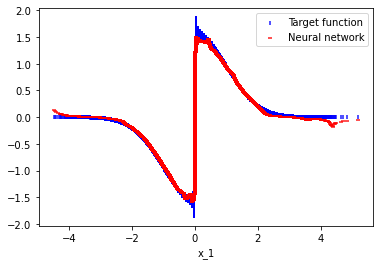}
\caption{Method 3: Layer by layer \& Adam}
\label{fig:tfnn_c3s}
\end{minipage}
\quad
\begin{minipage}{0.45\textwidth}
\includegraphics[width=0.95\linewidth]{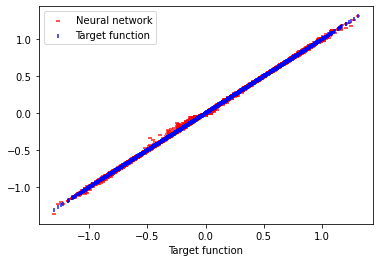}
\caption{Method 3: Layer by layer \& Adam. \newline
Correlation between neural network and target function: $0.999754$.}
\label{fig:tfnn_c3f}
\end{minipage}
\caption{Target function $f_1$ in dimension $d=10$, and approximating neural network. The subfigures to the left show a slice of the functions along the $x_1$-axis and the figures to the right show the function values plotted against the target function values.}
\label{fig:target_fcn_and_nn_f1}
\end{figure}




\begin{figure}[t!]
\begin{minipage}{0.45\textwidth}
\centering
\includegraphics[width=0.95\linewidth]{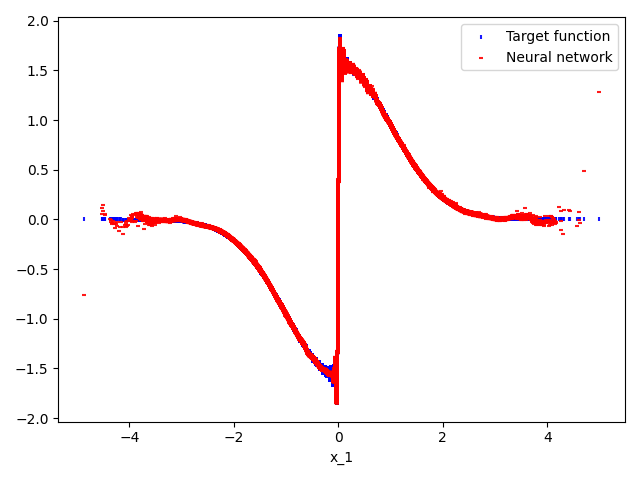}
\caption{Method 2: Xavier \& Adam}\label{fig:tfnn_c2_f2}
\end{minipage}
\quad
\begin{minipage}{0.45\textwidth}
\centering
\includegraphics[width=0.95\linewidth]{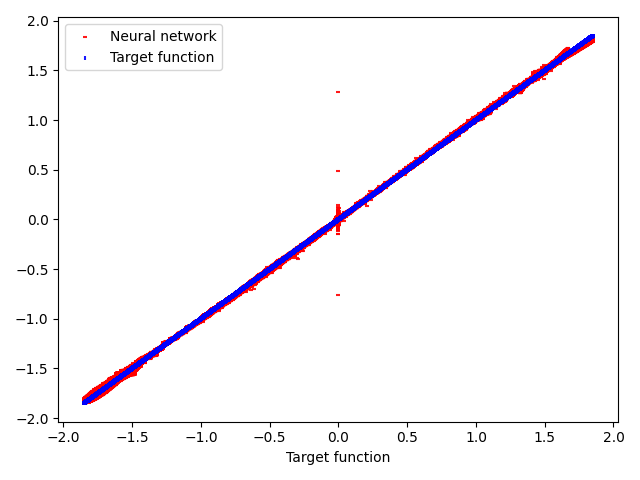}
\caption{Method 2: Xavier \& Adam}\label{fig:tfnn_c2_f2_yvy}
\end{minipage}

\vspace{0.4cm}

\begin{minipage}{0.45\textwidth}
\centering
\includegraphics[width=0.95\linewidth]{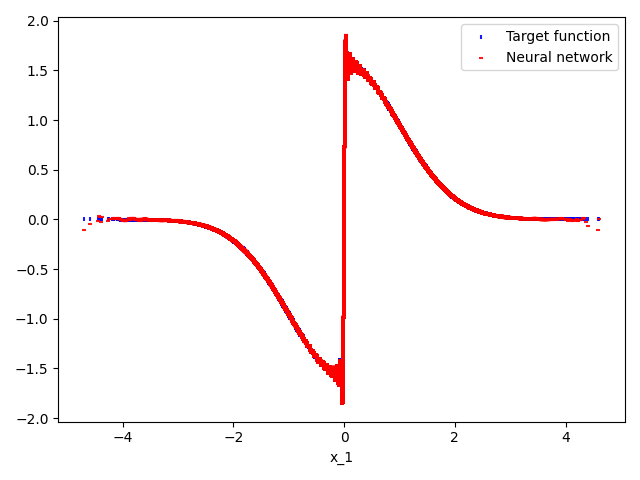}
\caption{Method 3: Layer by layer \& Adam}\label{fig:tfnn_c3_f2}
\end{minipage}
\quad
\begin{minipage}{0.45\textwidth}
\centering
\includegraphics[width=0.95\linewidth]{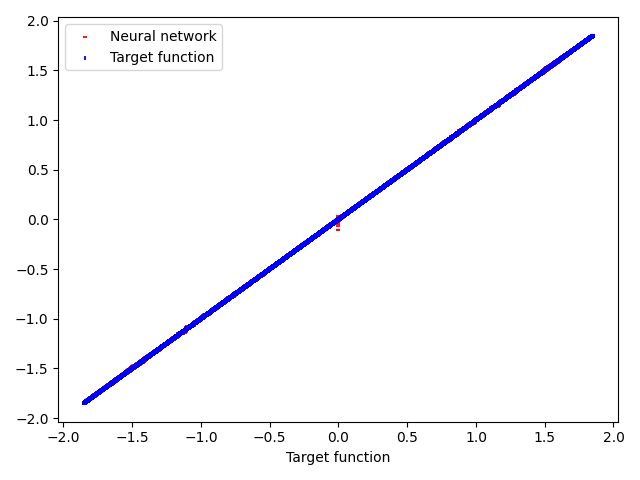}
\caption{Method 3: Layer by layer \& Adam}\label{fig:tfnn_c3_f2_yvy}
\end{minipage}
\caption{Target function $f_2$ in dimension $d=10$, and approximating neural network. The subfigures to the left show a slice of the functions along the $x_1$-axis and the figures to the right show the function values plotted against the target function values.
For Method 3, layer-by-ayer \& Adam, the neural network values  coincide with the target function values in the plot resolution.}
\label{fig:target_fcn_and_nn_f2}
\end{figure}

\begin{figure}[t!]
\centering
\includegraphics[width=0.99\textwidth]{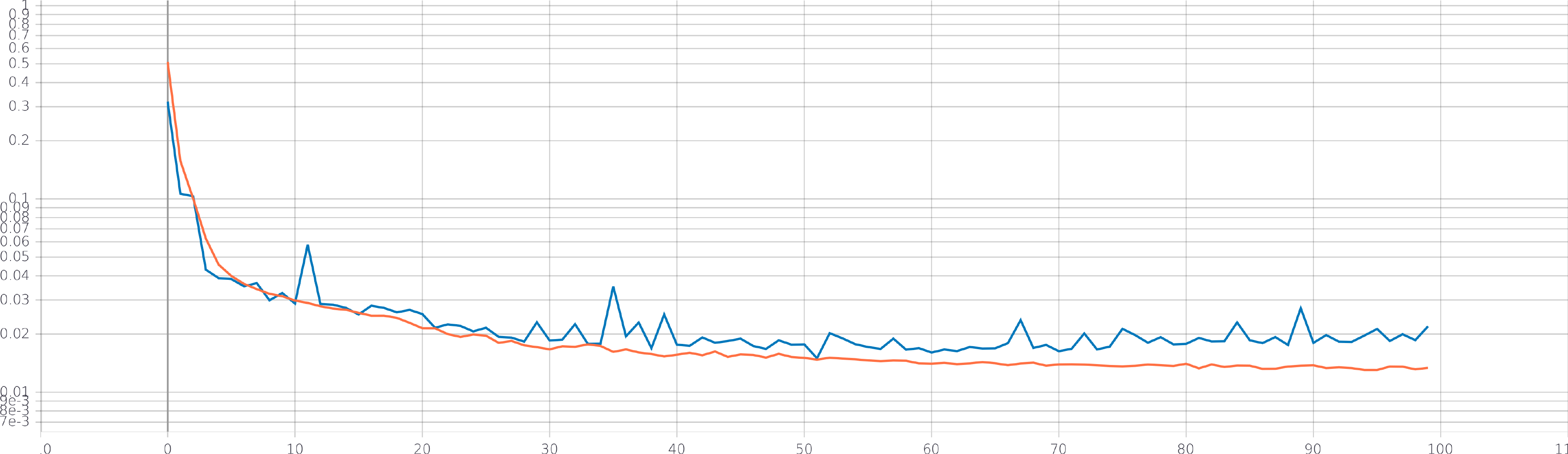}
\caption{The TensorBoard, see \cite{TensorBoard}, generated figure shows how the training error (orange) and the validation error (blue) decreases with respect to the number of epochs when running Method 2.}\label{fig:L5_dec_dt}
\end{figure}

\begin{table}
\caption{Summary of the parameter choices for the numerical experiments}
{
\centering
{
\bgroup
\renewcommand{\arraystretch}{1.4}
\setlength{\tabcolsep}{6pt}
\begin{tabular}{|l|c|c|c|c|c|c|c|}
\hline
                 & \multicolumn{7}{c|}{Regression}                                                                                                                                                                                                                                                                         \\ \hline
Method           & Method 1                     & \multicolumn{2}{c|}{Method 2}                                                                                        & \multicolumn{4}{c|}{Method 3}                                                                                                                     \\ \hline
Algorithm        & Algorithm 2                  & \multicolumn{2}{c|}{Xavier + Adam}                                                                                   & \multicolumn{4}{c|}{Algorithm 3}                                                                                                                  \\ \hline
$d$              & $3$                          & \multicolumn{2}{c|}{$10$}                                                                                            & \multicolumn{2}{c|}{$4$}                          & $10$                                       & $10$                                             \\ \hline
$KL$             & \begin{tabular}[c]{@{}c@{}}$5\times 2^i$,\\ $ i = 1,2,...,6$\end{tabular}  & $2560$                                                            & $1280$                                           & \multicolumn{2}{c|}{\begin{tabular}[c]{@{}c@{}}$5\times 2^i$,\\ $ i = 1,2,...,8$\end{tabular} } & $2560$                                     & $1280$                                           \\ \hline
$L$              & $1$ and $5$                  & \multicolumn{2}{c|}{$10$}                                                                                            & $1$                      & $5$                    & \multicolumn{2}{c|}{$10$}                                                                     \\ \hline
$N$              & $2\times 10^4$               & \multicolumn{6}{c|}{$10^6$}                                                                                                                                                                                                                                              \\ \hline
$\tilde{N}$      & $2\times 10^4$               & \multicolumn{6}{c|}{$10^6$}                                                                                                                                                                                                                                              \\ \hline
Target func.   & \multicolumn{2}{c|}{$f_1$}                                                                       & $f_2$                                            & \multicolumn{3}{c|}{$f_1$}                                                                     & $f_2$                                            \\ \hline
$N_1$            &                              &                                                                   &                                                  & \multicolumn{4}{c|}{$10^4$}                                                                                                                       \\ \hline
$\bar{M}$        & \multicolumn{2}{c|}{$10$}                                                                        & $20$                                             & $11$                     & $10$                   & $10$                                       & $10$                                             \\ \hline
$\gamma$         & $3d-2$                       &                                                                   &                                                  & \multicolumn{4}{c|}{$3d-2$}                                                                                                                       \\ \hline
${\hat \delta}$  & $1.1$                        &                                                                   &                                                  & \multicolumn{4}{c|}{$1.1$}                                                                                                                        \\ \hline
$M$              & $2000$                       &                                                                   &                                                  & \multicolumn{2}{c|}{$200$}                        & $400$                                      & $600$                                            \\ \hline
Batch size       &                              & \multicolumn{6}{c|}{$100$}                                                                                                                                                                                                                                               \\ \hline
Num. epochs &                              & $100$                                                             & $200$                                            & \multicolumn{2}{c|}{$150$}                        & $100$                                      & $200$                                            \\ \hline
$\check\delta$   & $0.5\times 2.4^2/d$          &                                                                   &                                                  & \multicolumn{4}{c|}{$0.5\times 2.4^2/d$}                                                                                                          \\ \hline
$\Delta t_e$     &                              & \multicolumn{6}{c|}{$0.001$}                                                                                                                                                                                                                                             \\ \hline
$m$              & $1$                          &                                                                   &                                                  & \multicolumn{4}{c|}{$1$}                                                                                                                          \\ \hline
Figure           & $\ref{fig:case1_LK}$         & $\ref{fig:tfnn_c2s}$, $\ref{fig:tfnn_c2f}$, $\ref{fig:L5_dec_dt}$ & $\ref{fig:tfnn_c2_f2}, \ref{fig:tfnn_c2_f2_yvy}$ & \multicolumn{2}{c|}{$\ref{fig:case3_LK}$}         & $\ref{fig:tfnn_c3s}$, $\ref{fig:tfnn_c3f}$ & $\ref{fig:tfnn_c3_f2}, \ref{fig:tfnn_c3_f2_yvy}$ \\ \hline
Table            &                              & \multicolumn{2}{c|}{$\ref{table:stat_errs}$}                                                                         &                          &                        & \multicolumn{2}{c|}{$\ref{table:stat_errs}$}                                                  \\ \hline
\end{tabular}
\egroup
}
\label{table:paramsum}
}
\end{table}

\subsection{Summary of experimental results}
In Table~\ref{table:stat_errs} we see that Method 3, i.e., layer-by-layer \& Adam, produces a smaller  generalization error than Method 2, i.e. Adam initialized by Xavier, for the chosen number of epochs for both target functions $f_1$ and $f_2$. We note from 
Figure~\ref{fig:target_fcn_and_nn_f1} that the approximation Method 3 produces a better fit to the target function than Method 2. Both methods would give better results for larger values of $K$ and if ran for a larger number of epochs. It cannot be concluded that Method 2 would not eventually produce a generalization error of approximately the same value as Method 3.

In Figure~\ref{fig:case1_LK} the generalization error decreases as expected, namely as $\mathcal{O}((KL)^{-1})$, and we note that for a fixed value of $KL$, the generalization error in Method 1 is, also as expected, smaller for $L = 5$ than for $L = 1$.

Figure~\ref{fig:L5_dec_dt} shows how the error decreases with respect to the number of epochs for Adam iterations. For the specific run the number of epochs is chosen to be $100$ with the purpose to limit the computational time even though the error continues to decrease. The choice of the number of epochs to run is a balance between getting the smallest error and getting a reasonable computational time.

When running Method 3 for several independent instances and  several different values of $KL$ we note, from Figure~\ref{fig:case3_LK}, that when $L = 1$ the generalization error has much smaller spread between each instance compared to when $L = 5$. On the other hand, the generalization error is much smaller for each run when $L = 5$ than when $L = 1$.

\section{Conclusions}\label{sec_conclusion}
A mean field optimal control applied to deep residual networks based on random Fourier features is used to derive an optimal distribution of random Fourier features. This provides a smaller generalization error compared to the well known estimate in \cite{barron,barron:94} for a network with one hidden layer and the same total number of nodes. The insight provided by this result is used to construct a new layer-by-layer training algorithm. The layer-by-layer method samples an approximation of the optimal distribution using an adaptive Metropolis method, where the computed amplitudes are used in the Metropolis proposal step.
The layer-by-layer method is used as an initialization for a global optimizer, such as Adam.

\section*{Acknowledgment}
This research was supported by
Swedish Research Council grant 2019-03725. The work of P.P. was supported in part by the ARO Grant W911NF-19-1-0243.
The work of R.T. was partially supported by the KAUST Office of Sponsored Research (OSR) under Award numbers URF$/1/2281-01-01$, URF$/1/2584-01-01$ in the KAUST Competitive Research Grants Program Round 8 and the Alexander von Humboldt Foundation.


\bibliographystyle{plain}
\bibliography{ref}   

\end{document}